\documentclass[12pt,reqno]{amsart}
\usepackage{color}
\usepackage{todonotes,comment,mathrsfs}
\usepackage{stackengine}
\usepackage{latexsym}
\usepackage{bm}
\usepackage{amsmath,amssymb,dsfont,amsfonts,mathtools}
\usepackage{mathrsfs}
\usepackage{verbatim}
\usepackage{graphicx}
\usepackage{graphics}
\usepackage{geometry}
\usepackage{booktabs}
\usepackage[shortlabels]{enumitem}
\usepackage{xcolor}
\definecolor{olive}{rgb}{0.3, 0.4, .1}
\definecolor{fore}{RGB}{249,242,215}
\definecolor{back}{RGB}{51,51,51}
\definecolor{title}{RGB}{255,0,90}
\definecolor{dgreen}{rgb}{0.,0.6,0.}
\definecolor{gold}{rgb}{1.,0.84,0.}
\definecolor{JungleGreen}{cmyk}{0.99,0,0.52,0}
\definecolor{BlueGreen}{cmyk}{0.85,0,0.33,0}
\definecolor{RawSienna}{cmyk}{0,0.72,1,0.45}
\definecolor{Magenta}{cmyk}{0,1,0,0}

\usepackage[shortlabels]{enumitem}
\newtheorem{proposition}{Proposition}[section]
\newtheorem{theorem}{Theorem}[section]

\newtheorem{corollary}{Corollary}[section]
\newtheorem{lemma}{Lemma}[section]
\newtheorem{remark}{Remark}[section]

\newtheorem{assumption}{Assumption}[section]

\def\R{\mathbb R}

\def\D{\mathbb D}
\def\E{\mathbb E}

\numberwithin{equation}{section}

\title[Forward stochastic integration w.r.t. fractional Brownian motion]{Forward stochastic integration for adapted processes w.r.t. Riemann-Liouville fractional Brownian motion (Full version). }

\author{Paulo Henrique da Costa$^1$}
\author{Alberto Ohashi$^2$}
\author{Francesco Russo$^3$}

\address{$1$ Departamento de Matem\'atica, Universidade de Bras\'ilia, 13560-970, Bras\'ilia - Distrito Federal, Brazil}\email{phcosta@unb.br}
\address{$2$ Departamento de Matem\'atica, Universidade de Bras\'ilia, 13560-970, Bras\'ilia - Distrito Federal, Brazil}\email{ohashi@mat.unb.br}
\address{$3$ ENSTA Paris, Institut Polytechnique de Paris,
 Unit\'e de Math\'ematiques appliqu\'ees, 828, boulevard des Mar\'echaux, F-91120 Palaiseau, France}
 \email{francesco.russo@ensta.fr}

\date{December 1st 2025}

\begin{document}

\begin{abstract}
  This paper provides the time-dependent $L^2$-martingale representation of the forward stochastic integral where the driving noise is the Riemann-Liouville fractional Brownian motion with parameter $\frac{1}{2} < H < 1$ and the integrand is a square-integrable adapted process. As a by-product, we obtain the exact $L^2$-isometry of the forward stochastic integrals based on suitable
  conditions on time-dependent martingale representations of adapted integrands combined with the Nelson's stochastic derivative of the underlying Gaussian driving noise.   
\end{abstract}

\maketitle




\section{Introduction}

Stochastic integration and calculus is a classical and fundamental topic,
started by It\^o in \cite{Ito}.
If $X$ is a standard Brownian motion and $Y$
is a progressively measurable process on a probability space $(\Omega, \mathcal{F},\mathbb{P})$, then
the variance of the It\^o's integral is given by
\begin{equation}\label{eq:Vito}
  {\rm Var}\Big(\int_0^T Y dX\Big) = \mathbb{E}\int_0^T |Y_s|^2 ds,
  \end{equation}
  provided the right-hand side is finite for a finite terminal time
  $0 < T < \infty$. 

The isometry  \eqref{eq:Vito} admits well-known
 extensions to the case when an integrator $X$ is a local martingale.
 In this context $ds$ is replaced by 
 $d \langle X\rangle_s$, where $\langle X \rangle$
 is such that $X^2 - \langle X \rangle$ is a local martingale,
according to the Doob-Meyer decomposition. In this way,
  the right-hand side of \eqref{eq:Vito} is still
 the expectation of a Lebesgue-type integral.
 Generalizations of  \eqref{eq:Vito} are crucial for instance in
 studying fixed point theorems in the framework of
 stochastic differential equations and analysis  of numerical schemes.
 In general, if $X$ is not a semimartingale, then $\langle X \rangle$ is not available. 

The main goal of this paper is to
construct a stochastic integral $\int_0^T Y dX$,
when the {\it integrator} $X$ is a continuous Gaussian process
and the {\it integrand}  $Y$ is  progressively measurable. More importantly, we aim to construct $\int_0^T Y dX$ via a suitable generalization of
\eqref{eq:Vito}.

There are several attempts to define
stochastic integrals beyond semimartingales. Among them, we can list the following literature on this topic.
\begin{enumerate}
\item {\it The pure pathwise approach}.
 The stochastic integral $\int_0^T Y dX$ is considered as a
  limit in the sense of  Riemann-Stieltjes sums, 
  for all $\omega \in \Omega$. 
  The seminal article comes back to \cite{young} by L.C. Young,
  which, given $X$ having finite $p$-variation and $Y$
  finite $q$-variation with $\frac{1}{p} + \frac{1}{q} > 1$,
  gives sense to $\int_0^T Y dX$ and establishes proper
  continuity properties.
  This has been extended to the case of more irregular integrator $X$
  by the rough path integration theory, initiated by T. Lyons, see
  e.g. \cite{lyonsq} but the integral itself has been
  formulated by \cite{feyel2006, gubinelli2004}
  also via the development
  of the celebrated {\it sewing lemma}.
  An alternative approach to Young integral was proposed by M. Z\"ahle,
   \cite{zahle}.
\item {\it The Malliavin-Skorohod calculus approach.}
  When $X$ is a Gaussian process, one natural object for stochastic
  integration is the Skorohod divergence-integral, which exists
  for instance when $Y$ belongs
to the so called Sobolev-Watanabe space ($\D^{1,2}(H)$),
$H$ being the self-reproducing space.
In this case 
  $\int_0^T YdX$ (in the
  sense of non-anticipating Riemann-Stieltjes sums) exists when
the Skorohod-divergence integral exists and 
when $Y$ has a suitable trace property.  
In this case even 
 the variance can be calculated, see e.g. \cite{nualart2006},
\cite{decreusefond, alos2001stochastic, alos2003, kruk2007, krukrusso, OhashiRusso}.
Unfortunately the conditions $Y \in \D^{1,2}(H)$ is far from being optimal.
\item {\it The stochastic sewing lemma techniques}.
  The technique comes back to \cite{Le}, which has adapted
  the original pathwise sewing lemma of \cite{feyel2006, gubinelli2004}
  taking into account  some martingale features of the integrator.
  This has allowed to prove existence of stochastic integral via
  limit of non-anticipating Riemann-Stieltjes sums and $L^p$-upper bounds for the integral
  when  the integrand  $Y$ is of the form $Y= g(\cdot,X)$. We refer e.g. to \cite{PerMats, hairerLi} for the
  implementation of the stochastic sewing lemma techniques
  in the characterization of stochastic integral.
\end{enumerate}
We recall that in the classical case of (\ref{eq:Vito}), the It\^o stochastic integral coincides with the forward integral $\int_0^T Yd^-X$, interpreted in the sense of the stochastic calculus via regularization introduced by \cite{rv1}, for a large class of progressively measurable processes $Y$. More precisely, the It\^o integral $\int_0^T Y_tdX_t$ is the limit in probability of
\begin{equation}\label{fapp}
\int_0^T Y_s \frac{X_{s+\varepsilon} - X_s}{\varepsilon}ds,
\end{equation}
as $\epsilon \downarrow 0$. We refer the reader to the monograph \cite{Russo_Vallois_Book} for more details on this approach. We also stress that when $Y$ is continuous, then the forward integral can be also expressed as the limit of non-anticipating Riemann sums (see e.g. \cite{fo}).


In this paper, we concentrate on the case of forward integral
when $X$ is a fractional Brownian motion of Hurst index $\frac{1}{2} < H < 1$,
see \cite{NourdinBook} for a fairly complete monograph
on that process. We emphasize that this process is not a semimartingale.
In
\cite{OhashiRusso1}, the authors have evaluated the existence of the forward integral
$\int_0^T Y d^-X$ as an $L^2(\Omega)$ object,  
when $Y_t = g(t,X_t)$, where $g: [0,T] \times \R^d \rightarrow \R$
is a Borel function; we denominate this class of processes as {\it state dependent}. In \cite{OhashiRusso1}, the $L^2(\Omega)$-norm of 
$\int_0^T Y d^-X$ is evaluated in terms of
the intrinsic covariance of $X$. The present article aims to study the existence of
 $\int_0^T Y d^-X$ and fully characterize its $L^2(\Omega)$-structure, when $X= B$ is a Riemann-Liouville fractional Brownian motion with
 Hurst index $\frac{1}{2} < H < 1$, for a large class of path-dependent processes $Y$, adapted
 w.r.t. the canonical filtration $\mathbb{F}= (\mathcal{F}_t)_{t\ge 0}$ generated by $B$.
 Recall that $B$ can be expressed as a Volterra-type process
(see (\ref{kernel})) 
\begin{equation}\label{volterraIntro}
B_t = \int_{0}^{t}K(t,s)dW_s,
\end{equation}
where $W$ is a standard Brownian motion and $K(t,s)= \sqrt{2H}(t-s)^{H-\frac{1}{2}}$,
with the same canonical filtration. In order to reach our target, we establish the Volterra-type martingale representation of the forward stochastic integral $\int_0^T Yd^-B$ considered as a square-integrable random variable measurable w.r.t. $\mathcal{F}_T$, in terms of the martingale representation structure $(\varphi^{(1)}_Y, \varphi^{(2)}_Y)$ realizing 
\begin{eqnarray} \label{eq:YMartRep}
Y_t &=& \mathbb{E}[Y_t] + \int_{0}^t \varphi^{(1)}_Y(t,s)dW_s, \\ 
\varphi^{(1)}_Y(t, u) &=& \mathbb{E}[\varphi^{(1)}_Y(t,u)] + \int_0^u \varphi_Y^{(2)}(t,u; r)dW_r,
                          0 < u < t \le T. \nonumber
 \end{eqnarray}

This is done by inserting \eqref{eq:YMartRep} into the the approximation
(\ref{fapp}) and we take the $L^2(\Omega)$-limit as $\varepsilon \rightarrow 0^+$.

Under suitable conditions on $(Y,\varphi^{(1)}_Y)$, Theorem \ref{Koperator} and Corollary \ref{CorMTH} express the $\mathcal{F}_T$-measurable square-integrable variable 
$$\int_0^T Y_td^-B_t = \int_0^T \mathcal{K} Y(T,s)dW_s + \int_0^T \int_0^t \mathbb{E}[\varphi_Y^{(1)}(t,u)]\frac{\partial K}{\partial t}(t,u)dudt,$$ 
where 
\begin{equation}\label{Kintr}
\mathcal{K} Y (T,r) = \int_r^T \varphi^{(1)}_Y(t,r)\mathcal{D}_{r,t}B dt + \int_r^T \int_r^t \varphi^{(2)}_Y(t,v;r) \frac{\partial K}{\partial t}(t,v) dvdt
\end{equation}
$$
+ \int_r^T \mathbb{E} [Y_t|\mathcal{F}_r] \frac{\partial K}{\partial t}(t,r)dt.
$$
Theorem \ref{Koperator} and Corollary \ref{CorMTH} cover a large class of $\mathbb{F}$-adapted square-integrable processes $Y$ satisfying mild integrability conditions w.r.t. to driving noise $B$. See Sections \ref{S3} and \ref{S4} for examples.  

The random field 
\begin{equation}\label{nelsintr}
\mathcal{D}_{s,t}B:=\lim_{\epsilon \downarrow 0}\frac{1}{\epsilon}\mathbb{E}[B_{t+\epsilon}-B_t| \mathcal{F}_s]\quad L^2(\Omega)\text{-limit}, 0 < s < t,
\end{equation}
plays a central role in the exact $L^2(\Omega)$-limit of $\frac{1}{\epsilon}\int_0^T Y_t (B_{t+\epsilon} - B_t)dt$ because it allows us to get rid off of the ill-posedness of the asymptotic limit $\frac{1}{\epsilon}[B_{t+\epsilon} - B_t]$ as $\epsilon \downarrow 0$. In case $s=t$ and $X$ is an It\^o process $X$, $L^p(\Omega)$-limits of the form

$$\lim_{\epsilon \downarrow 0}\frac{1}{\epsilon}\mathbb{E}[X_{t+\epsilon}-X_t| \mathcal{G}_t],$$
where $\mathcal{G}_t = \sigma(X_s; s\le t)$, were extensively studied by Nelson \cite{nelson} in his dynamical theory of
Brownian diffusions.
Nelson-type stochastic derivatives (w.r.t. $\sigma$-algebra generated by $B$) were studied by \cite{darses} in the context of SDEs driven by fractional Brownian motion with Hurst parameter $\frac{1}{2} < H < 1$, see also
\cite{darses1,darsesbruno}.    

One should observe that the random field $\mathcal{D}_{t,t}B$ does not exist along the diagonal (see Proposition 10 in \cite{darsesbruno}) and it is absolutely essential to consider the limit (\ref{nelsintr}) for $s < t$. 
It is important to highlight that reduced Nelson-type stochastic derivatives of the form

\begin{equation}
\lim_{\epsilon\downarrow 0} \frac{1}{\epsilon}\mathbb{E}[B_{t+\epsilon} - B_t| B_s, B_t]\quad L^2(\Omega)\text{-limit}, 0 < s < t,
\end{equation}
for $\frac{1}{4} < H < 1$ with $H\neq \frac{1}{2}$, have already played a central role in the isometries obtained by \cite{OhashiRusso1,OhashiRusso2}. In these works, the state-dependence of the integrands $Y_\cdot = g (\cdot, B_\cdot )$ allows us to work intrinsically on the level of the Gaussian space of the driving Gaussian noise $B$. 
In particular, in \cite{OhashiRusso1} one shows
the existence, for this kind of processes,
of the forward integral $\int_0^t Y_s d^- B_s$ and we evaluate precisely its variance
in the term of the covariance function of $B$, when $g$ has polynomial growth.

In the present work, we go much beyond the state-dependent case treated in \cite{OhashiRusso1}. We investigate natural integrability conditions of a generic $\mathbb{F}$-adapted square-integrable integrand w.r.t. the driving noise $B$. In order to investigate the exact $L^2(\Omega)$-isometry of the forward integral in this general context, we need to invoke the classical It\^o calculus w.r.t. the underlying Brownian motion $W$.
In order to find the {\it largest possible class} of integrands $Y$, taking into account previous considerations when $Y$ is a state dependent case, one could think that the forward integral exists
for every bounded progressively measurable process $Y$.
It is important to mention (see Proposition 1.2 of \cite{rv1}) that the existence of the forward integral $\int_0^T Yd^-B$ for all bounded predictable processes $Y$ forces $B$ to be a semimartingale.
Therefore, in order to study the existence of the forward integral
(in  $L^2(\Omega)$) for a given  $\mathbb{F}$-adapted square-integrable integrand $Y$, we need to impose more delicate conditions
on the ''derivative'' $\varphi^{(1)}_Y$  w.r.t. the driving noise, coupled with the Nelson derivative (\ref{nelsintr}).

In case $Y$ is a deterministic function in the reproducing kernel Hilbert space of $B$,
then $\varphi^{(1)}_Y = \varphi^{(2)}_Y=0$, $\mathbb{E}[Y_t|\mathcal{F}_r] = Y_t$ and a direct consequence of Theorem \ref{Koperator} is the well-known identity    
\begin{eqnarray*}
\text{Var}\Big( \int_0^T Y_t d^-B_t \Big) &=& \int_0^T \Bigg|\int_r^T Y_t\frac{\partial K}{\partial t}(t,r)dt\Bigg|^2dr\\
  &=& \int_0^T \int_0^T Y_tY_s \frac{\partial^2 R}{\partial t \partial s}dsdt,
\end{eqnarray*} 
for the Wiener integral, see e.g. \cite{nualart2006} and Section 5 of \cite{kruk2007}. In case $Y$ is random, then the random field $\{\mathbb{E}[Y_t|\mathcal{F}_r]; r < t\}$ is not enough to describe the stochastic forward integral and the presence of $(\varphi^{(1)}_Y, \varphi^{(2)}_Y)$ encodes the exact regularity condition one needs to check in order to describe  $\int_0^T Y_t d^-B_t.$ 
We observe that the martingale representation provides the identities 
$$\|Y_t - \mathbb{E}[Y_t]\|_{L^2(\Omega)} = \| \varphi_Y^{(1)}(t,\cdot)\|_{L^2(\Omega\times[0,t])},$$
\begin{equation}\label{phi2eq}
\|\varphi_Y^{(1)}(t,s) - \mathbb{E}[\varphi_Y^{(1)}(t,s)] \|_{L^2(\Omega)} = \| \varphi_Y^{(2)}(t,s; \cdot)\|_{L^2(\Omega\times [0,s])},
\end{equation}
for $s < t$.

In the case $H = \frac{1}{2}$, as mentioned in \eqref{eq:Vito},
it is enough to work on the level of $t\mapsto Y_t$. Theorem \ref{Koperator} shows that the stochastic calculus w.r.t. fractional Brownian motion with $H > \frac{1}{2}$ requires a more intricate condition on the level of the martingale derivative $\varphi^{(1)}_Y$ rather than only on $t\mapsto Y_t$. This is illustrated by the components 
$$\int_r^T \varphi^{(1)}_Y(t,r)\mathcal{D}_{r,t}B dt, \int_r^T \int_r^t \varphi^{(2)}_Y(t,v;r) \frac{\partial K}{\partial t}(t,v) dvdt, $$
which mix $(\varphi^{(1)}_Y,\varphi^{(2)}_Y)$ with the weights $(\mathcal{D}B,\partial K)$ in (\ref{Kintr}). At this point, it is important to highlight that $\varphi^{(1)}_Y$ is actually the main component of the forward integral. Indeed, one can prove that the existence of $\int_r^T \int_r^t \varphi^{(2)}_Y(t,v;r) \frac{\partial K}{\partial t}(t,v) dvdt$ is ensured by a mild integrability condition on the first martingale derivative $\varphi^{(1)}_Y$ due to essentially (\ref{phi2eq}) and the fact that $\frac{\partial K}{\partial t} $   is deterministic. We refer here to Assumption
\ref{I4}, Lemma \ref{bassa3} and Corollary \ref{CorMTH}.

Interestingly, we stress that in the state-dependent case $Y_\cdot = g(\cdot, B_\cdot)$, it is known that only integrability condition on the level of $g$ is enough as mentioned in \cite{OhashiRusso1}. This article shows that in the general path-dependent case, integrability condition on the level of $t\mapsto Y_t$ is not enough which also coincides with Proposition 1.2 in \cite{rv1} as pointed out above.    

It is important to highlight that the present work establishes the Volterra-type martingale representation of the forward integral and the $L^2(\Omega)$-isometry is a by-product of such representation. We achieve these results by exploiting
the underlying Brownian motion generating the Riemann-Liouville fractional Brownian motion. In \cite{OhashiRusso1}, the authors establish the isometry via an intrinsic argument on the level of the Gaussian space of the driving Gaussian noise $B$. Lastly, we mention that our framework goes much beyond the classical philosophy of Malliavin calculus based on the well-known decomposition in terms of the Skorohod operator plus the trace of the Malliavin derivative. In this direction, we refer the reader to e.g. \cite{alos2003,Russo_Vallois_Book, kruk2007, decreusefond, biaginibook,nualart2006}. Sections \ref{S3} and \ref{S4} show a large class of examples which is out of the scope of Malliavin calculus.     

The article is organized as follows. Section \ref{Wienersec} presents the simplest case of the Paley-Wiener integral as a forward stochastic integral. Section \ref{S2} presents the main result of the paper. Sections \ref{S3} and \ref{S4} present a class of examples ranging from state-dependent to fully path-dependent integrands.  


\section{The Wiener integral case}\label{Wienersec}

As mentioned earlier, in this article, $0 < T < \infty$ is a positive finite terminal time. Throughout this note, we write $a\lesssim b$ for two positive quantities to express an estimate of the form
$a \le C b$, where $C$ is a generic constant which may differ from line to line.
If $a\lesssim b$ and $b\lesssim a$, we simply write $a \simeq b$.
If $\gamma$ is a parameter, then $a\lesssim_{\gamma} b$ means that $a \le C b$, where the constant $C$ depends on $\gamma$.
Moreover, whenever necessary, we will adopt the convention that processes $X$ defined on $[0,T]$ can be extended to $[0,\infty)$ by setting $X_t=0$ for $t>T$. We fix a probability space $\big( \Omega, \mathcal{F}, \mathbb{P}\big)$ carrying a standard Brownian motion $W$ which generates a complete filtration $\mathbb{F} := (\mathcal{F}_t)_{t \ge 0}$. Conditional expectation of a random variable $Z$ w.r.t. the filtration $\mathbb{F}$ will be denoted by 

$$\mathbb{E}^s [Z]:= \mathbb{E}\big[ Z | \mathcal{F}_s  \big],$$
for $ s \ge 0$. Let us denote 

\begin{equation}\label{kernel}
K(t,s):= \sqrt{2H}(t-s)^{H-\frac{1}{2}};~0 < s< t,
\end{equation}
and $K(t,s):=0$ for $s\ge t$. The Riemann-Liouville fractional Brownian motion (henceforth abbreviated by RLFBM) will be represented via the Volterra-type representation 
\begin{equation}\label{volterra}
B_t = \int_{0}^{t}K(t,u)dW_u, 
\end{equation}
for $H \in (\frac{1}{2},1)$.

\begin{lemma}\label{repSECDER}
By It\^o isometry, the covariance kernel $R$ of the RLFBM can be written as

$$R(s,t) = \int_{0}^{s\wedge t} K(t,u) K(s,u) du.$$ 

Therefore, if $\frac{1}{2} < H < 1$, then $K$ vanishes on the diagonal and hence 

$$\frac{\partial^2 R}{\partial t \partial s}(s,t) = \int_{0}^{s\wedge t}\frac{\partial K}{\partial t}(t,u) \frac{\partial K}{\partial s} (s,u) du,$$
for $ s \neq t$. Moreover, for $s \neq t$,
\begin{equation}\label{secR}
\frac{\partial^2 R}{\partial t \, \partial s}(t,s)
= 2H \left(H - 1/2\right)^2
|t-s|^{2H-2}\text{beta}_{\frac{t \wedge s}{\,t \vee s\,}}\!\left(H - 1/2,\, 2 - 2H\right),
\end{equation}
where $\text{beta}_x(a,b) = \int_0^x u^{a-1}(1-u)^{b-1}du$ is the incomplete beta function.
\end{lemma}
\begin{proof}
The first two assertions are immediate. We provide the details of (\ref{secR}). For $0 < s < t$, recall
\[
\frac{\partial^2 R}{\partial t\,\partial s}(s,t)
 = 2H\alpha^2 \int_0^s (s-u)^{\alpha-1}(t-u)^{\alpha-1}\,du,
\qquad \alpha = H-\tfrac12.
\]

Let $w = t-s > 0$ and change variables $x = s-u$:
\[
\int_0^s (s-u)^{\alpha-1}(t-u)^{\alpha-1}\,du
 = \int_0^{s} x^{\alpha-1}(x+w)^{\alpha-1}\,dx.
\]

Next, change the variables $z = x/w$ so that 
\[
\int_0^{s} x^{\alpha-1}(x+w)^{\alpha-1}\,dx
 = w^{2\alpha-1} \int_0^{s/(t-s)} z^{\alpha-1}(1+z)^{\alpha-1}\,dz.
\]

Now let $y = \dfrac{z}{1+z}$, so that $z = \dfrac{y}{1-y}$ and $dz = \dfrac{dy}{(1-y)^2}$.
This yields
\[
\int_0^{s} (s-u)^{\alpha-1}(t-u)^{\alpha-1}\,du
 = w^{2\alpha-1}\int_0^{s/t} y^{\alpha-1}(1-y)^{-2\alpha}\,dy.
\]

Recognizing the incomplete Beta function
\[
\text{beta}_x(p,q) = \int_0^x u^{p-1}(1-u)^{q-1}\,du,
\]
we obtain
\[
\int_0^{s} (s-u)^{\alpha-1}(t-u)^{\alpha-1}\,du
 = (t-s)^{2\alpha-1} \text{beta}_{s/t}(\alpha,\, 2-2H),
\]
since $2-2H = 1 - 2\alpha$. Hence, for $0 < s < t$,
\[
\frac{\partial^2 R}{\partial t\,\partial s}(s,t)
 = 2H\bigl(H-\tfrac12\bigr)^2 (t-s)^{2H-2}\,
\text{beta}_{s/t}\!\bigl(H-\tfrac12,\, 2-2H\bigr).
\]
By the symmetry of $R(s,t)$, the case $t<s$ follows by exchanging $s$ and $t$.
\end{proof}

\begin{remark}
The RLFBM with $\frac{1}{2} < H < 1$ has a covariance measure structure in the sense of \cite{kruk2007}. That is, the planar increments of the covariance kernel $R$ can be extended to a signed finite measure on $[0,T]^2$, see Definition 3.1 in \cite{kruk2007}. In the case of RLFBM with $\frac{1}{2} < H < 1$, the covariance measure is a finite positive absolutely continuous measure w.r.t. Lebesgue on $[0,T]^2$ whose density is $\frac{\partial^2 R}{\partial t \partial s}$. In case, $0 < H < \frac{1}{2}$, the measure is $\sigma$-finite. In this direction, see \cite{krukrusso,OhashiRusso}.    
\end{remark}

The increments of RLFBM can be written as

\begin{equation}\label{ortDECFBM}
B_{t}-B_s = \int_{0}^s \{K(t,u) - K(s,u)\}dW_u + \int_s^t K(t,u)dW_u,
\end{equation}
for $0\le s < t$. From (\ref{ortDECFBM}), we shall write 

\begin{eqnarray}\label{fdec}
\frac{1}{\epsilon}\big[B_{t+\epsilon}-B_t\big] &=&  \int_{0}^{t} \frac{1}{\epsilon}\Delta_\epsilon K (t,u)dW_u + \frac{1}{\epsilon}\int_{t}^{t+\epsilon} K(t+\epsilon, u) dW_u,
\end{eqnarray}
where we denote  
\begin{equation}\label{DeltaK}
\Delta_\epsilon K(t,u)= \Big\{K(t+\epsilon,u) - K(t,u)\Big\} \mathds{1}_{(0,t)}(u).
\end{equation}

Let us recall the definition of the forward stochastic integral in the sense of regularization (see e.g. \cite{Russo_Vallois_Book})

$$\int_0^T Y_t d^-B_t:=\lim_{\epsilon\downarrow 0}\frac{1}{\epsilon}\int_0^T Y_t [B_{t+\epsilon}-B_t]dt,$$
whenever this limit exists in $L^2(\Omega)$ for a measurable process $Y$.   

For pedagogical reasons, we now  discuss the class of deterministic functions for which the Wiener integral exists in the sense of regularization.
Let $|\mathcal{H}|$ be the space of Borel measurable functions $g:[0,T]\rightarrow \mathbb{R}$ such that 

$$\int_0^T \Bigg( \int_r^T |g_t| \frac{\partial K}{\partial t}(t,r)dt \Bigg)^2dr = \int_0^T\int_0^T |g_tg_s| \frac{\partial^2 R}{\partial t \partial s}(t,s)dsdt < \infty.$$



\begin{lemma}\label{remlemma}
Let $0<\alpha<\frac{1}{2}$ and $g\in L^p[0,T]$ with
\[
p>\frac{2}{1+2\alpha}.
\]
For $\varepsilon>0$, define
\[
F_\varepsilon(u):=\frac{1}{\varepsilon}\int_{u-\varepsilon}^{u}
g(t)\,\lvert t+\varepsilon-u\rvert^\alpha dt,
\]
for $0 < u < T$. Then
\[
\int_0^T \lvert F_\varepsilon(u)\rvert^2 du \rightarrow 0,
\]
as $\epsilon \downarrow 0$. 
\end{lemma}

\begin{proof}
Extend $g$ by $0$ outside $[0,T]$ (still denoted $g$), so that all integrals below may be taken over $\mathbb{R}$.
For $u\in\mathbb{R}$ and $\varepsilon>0$ we rewrite $F_\varepsilon$ by the change of variables $s=u-t$:
\[
F_\varepsilon(u)=\frac{1}{\varepsilon}\int_0^{\varepsilon} g(u-s)\,(\varepsilon-s)^{\alpha}\,ds
= (g * \kappa_\varepsilon)(u),
\]
where the kernel
\[
\kappa_\varepsilon(s)=\frac{1}{\varepsilon}(\varepsilon-s)^{\alpha}\mathbf 1_{[0,\varepsilon]}(s)
\]
is supported on $[0,\varepsilon]$. For any $1\le r<\infty$,
\[
\|\kappa_\varepsilon\|_{L^r(\mathbb{R})}^r
=\int_0^\varepsilon\Big(\tfrac{1}{\varepsilon}(\varepsilon-s)^\alpha\Big)^r ds
=\varepsilon^{-r}\int_0^\varepsilon (\varepsilon-s)^{\alpha r}ds
=\frac{\varepsilon^{\alpha r+1-r}}{\alpha r+1}.
\]
Hence
\begin{equation}\label{eq:kernel-norm}
\|\kappa_\varepsilon\|_{L^r(\mathbb{R})}
= C_{\alpha,r}\,\varepsilon^{\,\alpha-1+\frac{1}{r}},
\qquad
C_{\alpha,r}=(\alpha r+1)^{-1/r}.
\end{equation}
Fix $p$ as in the statement and choose $r\in[1,\infty]$ so that
\begin{equation}\label{eq:indices}
1+\frac{1}{2}=\frac{1}{p}+\frac{1}{r}
\quad\Longleftrightarrow\quad
\frac{1}{r}=\frac{3}{2}-\frac{1}{p}.
\end{equation}
By Young's convolution inequality on $\mathbb{R}$,
\[
\|F_\varepsilon\|_{L^2(\mathbb{R})}
=\|g*\kappa_\varepsilon\|_{L^2(\mathbb{R})}
\le \|g\|_{L^p(\mathbb{R})}\,\|\kappa_\varepsilon\|_{L^r(\mathbb{R})}.
\]
Using \eqref{eq:kernel-norm} and \eqref{eq:indices} we obtain
\[
\|F_\varepsilon\|_{L^2(\mathbb{R})}
\le C_{\alpha,r}\,\|g\|_{L^p(0,T)}\,
\varepsilon^{\,\alpha-1+\frac{1}{r}}
= C_{\alpha,p}\,\|g\|_{L^p(0,T)}\,
\varepsilon^{\,\alpha-1+\frac{3}{2}-\frac{1}{p}}.
\]
The exponent of $\varepsilon$ is positive exactly when
\[
\alpha - 1 + \frac{3}{2} - \frac{1}{p} > 0
\quad\Longleftrightarrow\quad
\frac{1}{p} < \alpha + \frac{1}{2}
\quad\Longleftrightarrow\quad
p>\frac{2}{1+2\alpha}.
\]
Under this hypothesis we conclude
\[
\|F_\varepsilon\|_{L^2(\mathbb{R})}\xrightarrow[\varepsilon\downarrow 0]{}0.
\]
Since $F_\varepsilon$ is supported in $[0,T]$ when $g$ is, the same limit holds in $L^2(0,T)$, i.e.
\[
\int_0^T |F_\varepsilon(u)|^2\,du\;\longrightarrow\;0.
\]
This completes the proof.
\end{proof}

\begin{remark}
The threshold $p_c=2/(1+2\alpha)$ is the critical value for the Young-type argument in Lemma \ref{remlemma}.
For $p\ge 2$ one can also use Young's inequality with $r=1$ and the fact that
$\|\kappa_\varepsilon\|_{L^1}=\varepsilon^{\alpha}/(\alpha+1)\to 0$,
followed by the embedding $L^p[0,T]\hookrightarrow L^2[0,T]$ on a finite measure space.
\end{remark}
Now, we make a fundamental use of Lemma \ref{remlemma}.  
\begin{proposition}\label{wiener}
If $g \in L^p[0,T]$ for $p > \frac{2}{1+2(H-\frac{1}{2})} = \frac{1}{H}$, then $g \in |\mathcal{H}|$ and 

$$\int_0^T g_td^-B_t = \int_0^T \Bigg( \int_r^T g_t\frac{\partial K}{\partial t}(t,r)dt \Bigg)dW_r.$$

\end{proposition}
\begin{proof}
Fix $p > \frac{1}{H}$ and let $g \in L^p[0,T]$. The RLFBM has a similar covariance structure to fractional Brownian motion. Indeed, since the incomplete beta function in (\ref{secR}) is bounded on $[0,1]$, then a similar estimate as described in Lemma 5.1.1 in \cite{nualart2006} holds true. In this case, 

$$\| g\|_{|\mathcal{H}|}\lesssim_H \|g\|_{L^{\frac{1}{H}}[0,T]} < \infty.$$
Observe that we have pointwise convergence 

$$\lim_{\epsilon\downarrow 0}g_t \frac{1}{\epsilon}\Delta_\epsilon K(t,r)=g_t  \frac{\partial K}{\partial t}(t,r)
$$
and mean value theorem yields 

\begin{equation}\label{b1}
\big|g_t  \frac{1}{\epsilon}\Delta_\epsilon K(t,r)\big|\le \big|g_t \big|\frac{\partial K}{\partial t}(t,r)
\end{equation}
for each $0 < r < t < T$. We will apply bounded convergence theorem in two steps. By assumption, for Lebesgue almost all $r$, 

$$\int_r^T |g_t| \frac{\partial K}{\partial t}(t,r) dt < \infty.$$
Hence, from (\ref{b1}) we may apply bounded convergence theorem to get

\begin{equation}\label{b2}
\lim_{\epsilon\downarrow 0}\Bigg( \int_r^T g_t \Big\{\frac{1}{\epsilon}\Delta_\epsilon K(t,r) - \frac{\partial K}{\partial t}(t,r)\Big\}dt\Bigg)^2=0, 
\end{equation}
for Lebesgue almost all $r$. Let us denote $\xi_\epsilon(t,r) = \{\frac{1}{\epsilon}\Delta_\epsilon K(t,r) - \frac{\partial K}{\partial t}(t,r)\}$ for $t > r$. By applying Fubini's theorem and (\ref{b1}), we observe 

$$\Bigg( \int_r^T g_t \Big\{\frac{1}{\epsilon}\Delta_\epsilon K(t,r) - \frac{\partial K}{\partial t}(t,r)\Big\}dt\Bigg)^2$$
$$ = \int_{[r,T]^2} g_{t_1} g_{t_2}\xi_\epsilon(t_1,r)\xi_\epsilon(t_2,r)dt_1dt_2$$
\begin{equation}\label{b3}
\le 4 \int_{[r,T]^2} \big|g_{t_1} g_{t_2}\big| \frac{\partial K}{\partial t_1}(t_1,r) \frac{\partial K}{\partial t_2}(t_2,r)dt_1dt_2,
\end{equation}
for $0\le r < T$. Fubini's theorem and the fact that $g \in L^p[0,T] \subset |\mathcal{H}|$ yield

$$\int_0^T\int_{[r,T]^2} \big|g_{t_1} g_{t_2}\big| \frac{\partial K}{\partial t_1}(t_1,r) \frac{\partial K}{\partial t_2}(t_2,r)dt_1dt_2 dr$$
\begin{equation}\label{b4} 
= \int_0^T\int_0^T \big|g_{t} g_{s}\big| \frac{\partial^2 R}{\partial t \partial s}(t,s)dtds< \infty.
\end{equation}
From (\ref{b2}), (\ref{b3}), (\ref{b4}) and It\^o isometry, we conclude
\begin{equation} \label{eq:WInt}
  \lim_{\epsilon\downarrow 0}\int_0^T \int_{0}^{t} \frac{1}{\epsilon}\Delta_\epsilon K (t,u)dW_u dt = \int_0^T \Bigg(\int^T_rg_t\frac{\partial K}{\partial t}(t,r)dt\Bigg) dW_r.
  \end{equation}
Now, from (\ref{fdec}), it remains to check that 

$$\lim_{\epsilon\downarrow 0}\int_0^T  \frac{1}{\epsilon}\int_{t}^{t+\epsilon} K(t+\epsilon, u) dW_u dt = 0$$
in $L^2(\Omega)$. By stochastic Fubini's theorem, we can write 

$$\int_0^T  \frac{1}{\epsilon}\int_{t}^{t+\epsilon} g_t K(t+\epsilon, u) dW_u dt = \int_0^T \Bigg(\frac{1}{\epsilon}\int_{u-\epsilon}^ug_t K(t+\epsilon,u)dt\Bigg) dW_u.$$
A direct application of Lemma \ref{remlemma} allows us to conclude the proof.   
\end{proof}
\begin{remark}\label{rmk:JFATudor}
  If $g$ is càdlàg  function, the result of Proposition \ref{wiener} was proved in Corollary 5.14 of \cite{kruk2007}.
    This was stated in the case when the finite covariance covariance measure $\vert t_1 - t_2 \vert^{2H-2} dt_1 dt_2$ on $[0,T]^2$
    was replaced by a generic finite signed mesure.
  \end{remark}
  
\begin{remark}
In the present paper, we will impose $L^2[0,T]$-integrability for the mean $t\mapsto \mathbb{E}[Y_t]$ of an adapted process $Y$. 
\end{remark}

In the sequel, let $L^2_a(\Omega\times [0,t])$ be the set of $\mathbb{F}$-progressively measurable process $\{\varphi(t,s); 0\le s\le t\}$ on $\Omega\times [0,t]$ such that 

$$\mathbb{E}\int_{0}^t |\varphi(t,s)|^2ds < \infty.$$

\section{The isometry for square-integrable adapted processes}

\label{S2}

Our methodology will be fully based on the classical It\^o's representation theorem w.r.t. a fixed Brownian motion $W$. If $\{Y_t; 0\le t\le T\}$ is an $\mathbb{F}$-adapted process such that $\mathbb{E}|Y_t|^2 < \infty$ for every $t \in [0,T]$, then for each $t \in [0,T]$, there exists a unique process $\varphi^{(1)}(t,\cdot) \in L^2_a(\Omega\times[0,t])$ which realizes 

\begin{equation}\label{martREP1}
Y_t = \mathbb{E}[Y_t] + \int_{0}^t \varphi^{(1)}_Y(t,s)dW_s,
\end{equation}
for $0\le t\le T$. Throughout this paper, we will assume the following standing assumption for all considered square-integrable $\mathbb{F}$-adapted processes.

\

\begin{assumption}\label{I0}
The martingale derivative 
$$(s,t,\omega)\mapsto \varphi^{(1)}_Y(t,s)(\omega)$$ 
is jointly measurable along the simplex $\{(s,t); 0 < s < t\le T\}$. For almost all (Lebesgue) $s < t$, we shall represent  

$$\varphi^{(1)}_Y(t, s) = \mathbb{E}[\varphi^{(1)}_Y(t,s)] + \int_0^s \varphi_Y^{(2)}(t,s; r)dW_r,$$
where we will suppose that $(r,s,t,\omega)\mapsto \varphi^{(2)}_Y(t,s;r)(\omega)$ is jointly measurable along the simplex $\{(r,s,t); r < s < t\le T \}$. 
\end{assumption}





The strategy is to exploit as much as possible the It\^o's representation (\ref{martREP1}) for a given square-integrable process $Y$ adapted w.r.t. filtration generated by the RLFBM which coincides with the filtration $\mathbb{F}$. Our philosophy is to view those processes as $\mathbb{F}$-adapted, i.e., they will be functionals w.r.t. a fixed standard Brownian motion $W$. For a given $\epsilon>0$, we split the increments of the RLFBM as described in (\ref{fdec}) and we write 

\begin{eqnarray}
\label{impdec} \int_0^T Y_t \frac{1}{\epsilon}[B_{t+\epsilon} - B_t]dt &=& \int_0^T Y_t \int_{0}^{t} \frac{1}{\epsilon}\Delta_\epsilon K(t,v)dW_ v dt\\
\nonumber&+& \int_0^T Y_t \frac{1}{\epsilon}\int_t^{t+\epsilon} K(t+\epsilon,u)dW_u dt\\
\nonumber                                                                                                                                                   &=&: I_1(\epsilon) + I_2(\epsilon).
\end{eqnarray}
\begin{remark} \label{rmk:Brownian}
Below we will show that $ I_2(\epsilon)\rightarrow 0$, so $I_1(\epsilon)$ will be another approximation of the forward integral.
We remark that the situation is quite different in the case $H= \frac{1}{2}$, where $I_1(\epsilon) = 0$
and $I_2(\epsilon)=  \int_0^T Y_t \frac{1}{\epsilon}[B_{t+\epsilon}-B_t]dt$ converges to the It\^o integral, see e.g. Theorem 5.1 of \cite{Russo_Vallois_Book}.
\end{remark}

In the sequel, we set $\bar{Y}_t:= Y_t-\mathbb{E}[Y_t]$ for $0\le t\le T$. 

\begin{lemma}\label{Drep1}
If  $\int_0^T\mathbb{E}|Y_t|^2dt < \infty$, then 

$$I_2(\epsilon)\rightarrow 0$$
in $L^{2}(\mathbb{P})$ as $\epsilon \downarrow 0$. 
\end{lemma}
\begin{proof}


By applying Fubini's stochastic theorem, we get

$$I_2(\epsilon) = \int_0^T \frac{1}{\epsilon}\int_t^{t+\epsilon} Y_t K(t+\epsilon,u)dW_u dt = \int_0^T \frac{1}{\epsilon}\int_{r-\epsilon}^r Y_t K(t+\epsilon,r)dtdW_r.$$
By applying Jensen's inequality, we get 

\begin{eqnarray*}
\mathbb{E}\int_0^T\Bigg|\frac{1}{\epsilon}\int_{r-\epsilon}^r Y_tK(t+\epsilon,r)dt\Bigg|^2dr&\le&  \int_0^T \frac{1}{\epsilon}\int_{r-\epsilon}^r \mathbb{E}| Y_t|^2 (t+\epsilon -r)^{2H-1}dtdr\\
&\le& \epsilon^{2H-1} \int_0^T  \frac{1}{\epsilon}\int_{r-\epsilon}^r \mathbb{E}| Y_t|^2dtdr\\
&\le& \int_0^T\mathbb{E}| Y_t|^2dt \epsilon^{2H-1}\rightarrow 0,
\end{eqnarray*}
as $\epsilon \downarrow 0$.

\end{proof}
At this point, we can concentrate on the convergence of $I_1(\epsilon)$.
Let us denote 
$$
\mathcal{D}^\epsilon_{s,t}B:=\frac{1}{\epsilon}\mathbb{E}^s[B_{t+\epsilon}-B_t],
$$
for $s < t$ and $\epsilon>0$. By using representation (\ref{volterra}), we
observe that

$$\mathcal{D}^\epsilon_{s,t}B = \frac{1}{\epsilon}\int_0^s\Delta_\epsilon K(t,u)dW_u,$$
for $s < t$.

By using integration by parts in (\ref{impdec}) and stochastic Fubini's theorem, we can express
$$ I_1(\epsilon) = I_{11}(\epsilon) + I_{12}(\epsilon),$$
where
\begin{eqnarray*}
  I_{11}(\epsilon) &=&  \int_0^T \int_0^t \E[Y_t] \frac{1}{\epsilon}\Delta_\epsilon K(t,u) dW_udt \\
  I_{12}(\epsilon) &=&  \int_0^T \int_0^t {\bar Y}_t  \frac{1}{\epsilon}\Delta_\epsilon K(t,u) dW_udt.
                         \end{eqnarray*}

Concerning $I_{12}(\epsilon)$  we have                     
       \begin{eqnarray}
  \nonumber
  I_{12}(\epsilon)   &=&
 \int_0^T \int_0^t \varphi^{(1)}_Y(t,u)dW_u \int_0^t \frac{1}{\epsilon}\Delta_\epsilon K(t,u) dW_udt \\
  \nonumber &=&
                \int_0^T \int_0^t\frac{1}{\epsilon}\Delta_\epsilon K(t,u)\varphi^{(1)}_Y(t,u)dudt\\
\nonumber&+& \int_0^T \int_r^T \Bigg(\int_0^r \frac{1}{\epsilon}\Delta_\epsilon K(t,u) dW_u\Bigg) \varphi^{(1)}_Y(t,r)dtdW_r\\
 \label{pq1}
                     &+& \int_0^T \int_r^T \mathbb{E}^r[\bar{Y}_t] \frac{1}{\epsilon}\Delta_\epsilon K(t,r) dt dW_r.
\end{eqnarray}
Consequently

\begin{eqnarray*}
I_{1}(\epsilon)  &=&
 \int_0^T \int_0^t \varphi^{(1)}_Y(t,u)dW_u \int_0^t \frac{1}{\epsilon}\Delta_\epsilon K(t,u) dW_udt \\
   &=&
                \int_0^T \int_0^t\frac{1}{\epsilon}\Delta_\epsilon K(t,u)\varphi^{(1)}_Y(t,u)dudt\\
&+& \int_0^T \int_r^T \Bigg(\int_0^r \frac{1}{\epsilon}\Delta_\epsilon K(t,u) dW_u\Bigg) \varphi^{(1)}_Y(t,r)dtdW_r\\
                     &+& \int_0^T \int_r^T \mathbb{E}^r[Y_t] \frac{1}{\epsilon}\Delta_\epsilon K(t,r) dt dW_r.
\end{eqnarray*}
Now, using the martingale representation for $\varphi^{(1)}_Y$, we can express 

\begin{eqnarray}
\label{pq2}\int_0^T \int_0^t\frac{1}{\epsilon}\Delta_\epsilon K(t,u)\varphi^{(1)}_Y(t,u)dudt &=& \int_0^T \int_0^t \mathbb{E}[\varphi^{(1)}_Y(t,u)]\frac{1}{\epsilon}\Delta_\epsilon K(t,u)dudt\\
\nonumber&+& \int_0^T \int_0^t \int_0^u \varphi^{(2)}_Y(t,u; v)dW_v \frac{1}{\epsilon}\Delta_\epsilon K(t,u)dudt.
\end{eqnarray}

By using stochastic Fubini's theorem again in (\ref{pq2}) and combining decompositions (\ref{pq1}) and (\ref{pq2}), we get

\begin{eqnarray}
  \nonumber
  I_1(\epsilon) &=&\int_0^T \int_0^t \mathbb{E}[\varphi^{(1)}_Y(t,u)]\frac{1}{\epsilon}\Delta_\epsilon K(t,u)dudt\\
\nonumber& + &  \int_0^T \Bigg\{ \int_r^T \varphi^{(1)}_Y(t,r)\mathcal{D} ^\epsilon_{r,t}B dt + \int_r^T \int_r^t \varphi^{(2)}_Y(t,v;r)\frac{1}{\epsilon}\Delta_\epsilon K(t,v)dvdt\\
  \label{Kepsilon}&+& \int_r^T \mathbb{E}^r [Y_t] \frac{1}{\epsilon}\Delta_\epsilon K(t,r)dt\Bigg\}dW_r.
\end{eqnarray}
\begin{remark} \label{rmk:Det}
  If  $Y$ is deterministic then $\varphi^{(1)}_Y  =  \varphi^{(2)}_Y = 0$ and $Y_t = \E^r[Y_t]$, so that
  the integral by  regularization will coincide with the Wiener integral  
  
$$\int_0^T \Bigg( \int_r^T Y_t \frac{\partial K}{\partial t}(t,r)dt \Bigg)dW_r,$$
by Proposition \ref{wiener}. 
\end{remark}

In the sequel, we will make a fundamental use of the so-called Nelson's derivative

\begin{equation}\label{nelson}
\mathcal{D}_{r,t}B:=\lim_{\epsilon\downarrow 0} \mathcal{D}^\epsilon_{r,t}B~\text{in}~L^2(\Omega),
\end{equation}
for $0 < r < t$. 

\begin{remark}
We can represent the Nelson's derivative by

$$\mathcal{D}_{r,t}B = \int_0^r \frac{\partial K}{\partial t}(t,u)dW_u,$$
for $0 < r < t$. It is important to stress that $\mathcal{D}_{t,t}B$ does not exists along the diagonal! For more details and further discussion, see e.g. \cite{darses}.  
\end{remark}

Let us define 

\begin{eqnarray*}
\mathbf{D}_{x_1,x_2;r}B&:=& \int_0^r \mathcal{D}_{s,x_1}B d\mathcal{D}_{s,x_2}B\\
&=&\int_0^r \mathcal{D}_{s,x_1}B \frac{\partial K}{\partial x_2} (x_2,s)dW_s
\end{eqnarray*}
and 
$$\mathbf{D}^\epsilon_{x_1,x_2;r}B:=\int_0^r \big(\mathcal{D}^\epsilon_{s,x_1}B\big) \frac{1}{\epsilon}\Delta_\epsilon K(x_2,s)dW_s$$
for $r < x_1\wedge x_2$, $x_1\neq x_2$ and $\epsilon \in(0,1)$.

\begin{lemma}\label{Depsilonbounds}

$$\sup_{\epsilon \in (0,1)}\big\|\mathbf{D}^\epsilon_{x_1,x_2;r}B\big\|^q_{L^q(\Omega)}\lesssim_{q,H} \int_0^r (x_1-s)^{q(H-1)}(x_2-s)^{q(H-\frac{3}{2})}ds,$$
for $q>2$ and 

$$\sup_{\epsilon \in (0,1)}\big\|\mathbf{D}^\epsilon_{x_1,x_2;r}B\big\|^q_{L^q(\Omega)}\lesssim_{q,H} \Bigg(\int_0^r (x_1-s)^{2(H-1)}(x_2-s)^{2(H-\frac{3}{2})}ds\Bigg)^{\frac{q}{2}},$$
for $1\le q \le 2$ and $r < x_1\wedge x_2$ with $x_1\neq x_2$.  
\end{lemma}

\begin{remark}\label{supq2}
Observe that 
$$\mathbf{D}^\epsilon_{x_1,x_2;r}B=\int_0^r \mathcal{D}^\epsilon_{s,x_1}Bd \mathcal{D}^\epsilon_{s,x_2}B$$
for $r < x_1\wedge x_2 \le T$. Moreover, in case $q=2$, we have 

\begin{eqnarray*}
\sup_{\epsilon \in (0,1)}\big\|\mathbf{D}^\epsilon_{x_1,x_2;r}B\big\|^2_{L^2(\Omega)}&\le& \big\|\mathbf{D}_{x_1,x_2;r}B\big\|^2_{L^2(\Omega)}\\
&=& \int_0^r (x_1-s)^{2(H-1)}(x_2-s)^{2(H-\frac{3}{2})}ds,
\end{eqnarray*}
for $r < x_1\wedge x_2 \le T$ with $x_1\neq x_2$.  
\end{remark}

\begin{proof}
Fix $r < x_1 \wedge x_2$ with $x_1\neq x_2$. For $q\ge 2$, a direct application of Burkholder-Davis-Gundy and Jensen inequalities yield
\begin{eqnarray}
\nonumber\mathbb{E}|\mathbf{D}^\epsilon_{x_1,x_2;r}B|^q&\lesssim_q& \mathbb{E}\Bigg( \int_0^r \Big| \mathcal{D}^\epsilon_{s,x_1}B \frac{1 }{\epsilon} \Delta_\epsilon K(x_2,s)\Big|^2ds\Bigg)^{\frac{q}{2}}\\
\label{est1}&\le& \mathbb{E}\int_0^r \Big| \mathcal{D}^\epsilon_{s,x_1}B \frac{1 }{\epsilon} \Delta_\epsilon K(x_2,s) \Big|^qds.
\end{eqnarray}
Again a direct application of Burkholder-Davis-Gundy inequality and Mean Value Theorem yield

\begin{equation}\label{est2}
\mathbb{E}|\mathcal{D}^\epsilon_{s,x_1}B|^q\lesssim_q\Bigg(\int_0^s(x_1-u)^{2H-3}du \Bigg)^{\frac{q}{2}},
\end{equation}
and 

\begin{equation}\label{est3}
\frac{1 }{\epsilon} \Delta_\epsilon K(x_2,s)\lesssim_{H,T} |\frac{\partial K}{\partial t}(x_2,s)|,
\end{equation}
for every $s < r < x_1\wedge x_2$. Then summing up (\ref{est1}), (\ref{est2}) and (\ref{est3}), we have 

\begin{eqnarray}\label{est4}
\mathbb{E}|\mathbf{D}^\epsilon_{x_1,x_2;r}B|^q&\lesssim_{q,H,T}&\int_0^r \Bigg( \int_0^s (x_1-u)^{2H-3}du \Bigg)^{\frac{q}{2}}(x_2-s)^{q(H-\frac{3}{2})}ds\\
\nonumber&\le& \int_0^r (x_1-s)^{q(H-1)} (x_2-s)^{q(H-\frac{3}{2})}ds,
\end{eqnarray}
for every $q\ge 2$. In case $1\le q < 2$, we apply H\"older's inequality w.r.t. the probability measure $\mathbb{P}$

\begin{eqnarray}
\nonumber\mathbb{E}|\mathbf{D}^\epsilon_{x_1,x_2;r}B|^q&\lesssim_q& \mathbb{E}\Bigg( \int_0^r \Big| \mathcal{D}^\epsilon_{s,x_1}B \frac{1 }{\epsilon} \Delta_\epsilon K(x_2,s)\Big|^2ds\Bigg)^{\frac{q}{2}}\\
\label{est5}&\le& \Bigg(\mathbb{E}\int_0^r \Big| \mathcal{D}^\epsilon_{s,x_1}B \frac{1 }{\epsilon} \Delta_\epsilon K(x_2,s) \Big|^2ds\Bigg)^{\frac{q}{2}}.
\end{eqnarray}
It\^o's isometry and Mean Value Theorem yield 

\begin{equation}\label{est6}
\mathbb{E}|\mathcal{D}^\epsilon_{s,x_1}B|^2\lesssim_T \int_0^s(x_1-u)^{2H-3}du,
\end{equation}
Then summing up (\ref{est5}) and (\ref{est6}), we have 

\begin{eqnarray*}
\mathbb{E}|\mathbf{D}^\epsilon_{x_1,x_2;r}B|^q&\lesssim_{q,H,T}&\Bigg(\int_0^r \int_0^s (x_1-u)^{2H-3}du (x_2-s)^{2(H-\frac{3}{2})}ds\Bigg)^{\frac{q}{2}}\\
\nonumber&\lesssim_{q,H}& \Bigg(\int_0^r (x_1-s)^{2H-2} (x_2-s)^{2H-3}ds\Bigg)^{\frac{q}{2}}. 
\end{eqnarray*}
for every $1\le q< 2$.
\end{proof}
\begin{remark}\label{intDeps}
By using Lemma \ref{Depsilonbounds},
  one can easily check that 
$$\int_{[0,T]^2}\int_0^{x_1\wedge x_2} \sup_{\epsilon \in (0,1)}\big\|\mathbf{D}^\epsilon_{x_1,x_2;r}B\big\|^q_{L^q(\Omega)}drdx_1dx_2 < \infty,$$
for every $1\le q < \frac{2}{\frac{3}{2}-H} $. 
\end{remark}

Let us fix $\frac{1}{2} < H < 1$. In order to establish the isometry for the forward stochastic integral w.r.t. a square-integrable adapted process, we will
require that the process $Y$ fulfills the following assumptions.

\

\begin{assumption}\label{I1}
  \begin{equation}
 \int_0^T \|Y_t\|^2_{L^2(\Omega)}dt < \infty.
\end{equation}
\end{assumption}

\

\  

\begin{assumption} \label{I3}
  \begin{equation}\
\int_{[0,T]^2}\int_0^{x_1\wedge x_2} \|\varphi^{(1)}_Y(x_1,r)\varphi^{(1)}_Y(x_2,r)\|_{L^p(\Omega)}  \sup_{\epsilon \in (0,1)}\| \mathbf{D}^\epsilon_{x_1,x_2;r}B\|_{L^q(\Omega)} drdx_1dx_2<\infty,
\end{equation}
for two conjugate exponents $1< p \le \infty$, $1\le q < \infty$. 
\end{assumption}
\

\begin{assumption} \label{I4}
\begin{equation}\label{ass4}
\mathbb{E}\int_0^T\Bigg(\int_r^T\int_r^t  |\varphi^{(2)}_Y(t,v;r)|\frac{\partial K}{\partial t}(t,v) dv dt\Bigg)^2dr < \infty. 
\end{equation}
\end{assumption}

\

\begin{assumption} \label{I5}
  \begin{equation}
  \int_0^T \int_0^t\Big|\mathbb{E}[\varphi^{(1)}_Y(t,u)]\Big| \frac{\partial K}{\partial t}(t,u)dudt < +\infty. 
\end{equation}
\end{assumption}

\begin{remark}\label{modHtrem}
Observe that Assumption \ref{I1} implies 

\begin{equation}\label{modHtensor}
\int_0^T\int_0^T \big\| Y_{x_1}\big\|_{L^2(\Omega)} \big\| Y_{x_2} \big\|_{L^2(\Omega)} \frac{\partial^2 R}{\partial x_1 \partial x_2}(x_1,x_2)dx_1dx_2 <\infty,
\end{equation}
see Proposition \ref{wiener}. 
\end{remark}

The next lemma provides an upper bound related to Assumption \ref{I4} in terms of the first martingale derivative $\varphi^{(1)}_Y$. This is indeed not surprising in view of relation (\ref{pq2}). In the sequel, we recall that $\bar{\varphi}^{(1)}_Y = \varphi^{(1)}_Y - \mathbb{E}[\varphi_Y^{(1)}]$. 

\begin{lemma}\label{bassa3}
There exists a constant $C$ which only depends on $T$ such that 
\begin{equation}\label{cond5}
\mathbb{E}\int_0^T\Bigg(\int_r^T\int_r^t  |\varphi^{(2)}_Y(t,v;r)|\frac{\partial K}{\partial t}(t,v) dv dt\Bigg)^2dr
\end{equation}
$$\le C \int_J \big\| \mathbb{E}^{v_1\wedge v_2} \big[ \bar{\varphi}^{(1)}_Y(x_1,v_1) \big] \big\|_{L^2(\Omega)} \big\| \mathbb{E}^{v_1\wedge v_2} \big[ \bar{\varphi}^{(1)}_Y(x_2,v_2) \big] \big\|_{L^2(\Omega)}\mu(dv_1dv_2dx_1dx_2),$$
where 

\begin{equation}\label{mumeasure}
\mu(dv_1dv_2dx_1dx_2):= \frac{\partial K}{\partial x_1}(x_1,v_1)\frac{\partial K}{\partial x_2}(x_2,v_2)dv_1dv_2dx_1dx_2
\end{equation}
and 

\begin{equation}\label{Jset}
J = \{(v_1,v_2,x_1,x_2); v_1 < x_1, v_2 < x_2, v_1\wedge v_2 < x_1\wedge x_2 \}.
\end{equation} 
\end{lemma}
\begin{proof}
Let $I=\mathbb{E}\int_0^T\Big(\int_r^T\int_r^t  |\varphi^{(2)}_Y(t,v;r)|\frac{\partial K}{\partial t}(t,v) dv dt\Big)^2dr$. Fubini's theorem yields 

$$I = \mathbb{E}\int_J \int_0^{v_1\wedge v_2} |\varphi^{(2)}_Y(x_1,v_1;r) \varphi^{(2)}_Y(x_2,v_2;r)|dr \mu(dv_1dv_2dx_1dx_2) $$
Cauchy-Schwarz's inequality yields 

$$\mathbb{E}|\varphi^{(2)}_Y(x_1,v_1;r) \varphi^{(2)}_Y(x_2,v_2;r)|\le \|\varphi^{(2)}_Y(x_1,v_1;r)\|_{L^2(\Omega)} \|\varphi^{(2)}_Y(x_2,v_2;r)\|_{L^2(\Omega)},$$
for $r < v_1\wedge v_2$ and $(v_1,v_2,x_1,x_2) \in J$ and another application of Cauchy-Schwarz's inequality yield

$$\mathbb{E}\int_0^{v_1\wedge v_2} |\varphi^{(2)}_Y(x_1,v_1;r) \varphi^{(2)}_Y(x_2,v_2;r)|dr$$
$$\le \Bigg( \int_0^{v_1\wedge v_2} \|\varphi^{(2)}_Y(x_1,v_1;r)\|^2_{L^2(\mathbb{P})}dr \Bigg)^{\frac{1}{2}} \Bigg( \int_0^{v_1\wedge v_2} \|\varphi^{(2)}_Y(x_2,v_2;r)\|^2_{L^2(\mathbb{P})}dr \Bigg)^{\frac{1}{2}}$$
for $(v_1,v_2,x_1,x_2) \in J$. From Burkholder-Davis-Gundy inequality, we know that 

$$\int_0^{v_1\wedge v_2} \|\varphi^{(2)}_Y(x_1,v_1;r)\|^2_{L^2(\Omega)}dr\simeq
\big\| \mathbb{E}^{v_1\wedge v_2}\big[ \bar{\varphi}^{(1)}_Y(x_1,v_1)\big]\big\|^2_{L^2(\Omega)}$$
and 

$$\int_0^{v_1\wedge v_2} \|\varphi^{(2)}_Y(x_2,v_2;r)\|^2_{L^2(\Omega)}dr\simeq \big\| \mathbb{E}^{v_1\wedge v_2}\big[ \bar{\varphi}^{(1)}_Y(x_2,v_2)\big]\big\|^2_{L^2(\Omega)},$$
for $(v_1,v_2,x_1,x_2) \in J$. This concludes the proof.

\end{proof}
Inspired by Lemma \ref{bassa3}, we will also consider the following assumption.

\begin{assumption} \label{I6}
For $\bar{\varphi}^{(1)}_Y = \varphi^{(1)}_Y - \mathbb{E}[\varphi_Y^{(1)}]$,  

\begin{equation}
\int_J \big\| \mathbb{E}^{v_1\wedge v_2} \big[ \bar{\varphi}^{(1)}_Y(x_1,v_1) \big] \big\|_{L^2(\Omega)} \big\| \mathbb{E}^{v_1\wedge v_2} \big[ \bar{\varphi}^{(1)}_Y(x_2,v_2) \big] \big\|_{L^2(\Omega)}\mu(dv_1dv_2dx_1dx_2)<\infty,
\end{equation}
where $\mu$ and $J$ are given by (\ref{mumeasure}) and (\ref{Jset}), respectively. 
\end{assumption}
We are now ready to state the main result of this article.  

\begin{theorem}\label{Koperator}
Suppose that $Y$ satisfies the Assumptions \ref{I0}, \ref{I1},
  \ref{I3}, \ref{I4} and \ref{I5}. Then,   

\begin{eqnarray}
\label{can}\int_0^T Y_sd^-B_s &=& \int_0^T \int_0^s \mathbb{E}[\varphi^{(1)}_Y(s,u)]\frac{\partial K}{\partial s} (s,u)duds\\
\nonumber&+& \int_0^T \mathcal{K}Y(T,r)dW_r,
\end{eqnarray}
where

\begin{eqnarray}
\label{calKop}\mathcal{K}Y(T,r)&:=& \int_r^T \Bigg\{\mathbb{E}^r[Y_t] \frac{\partial K}{\partial t} (t,r) + \varphi_Y^{(1)}(t,r)\mathcal{D}_{r;t}B\\
\nonumber&+& \int_r^t \varphi^{(2)}_Y(t,v; r) \frac{\partial K}{\partial t}(t,v)dv\Bigg\}dt,
\end{eqnarray}
for $0\le r< T$.
\end{theorem}

Theorem \ref{Koperator} establish the Volterra-type martingale representation of $\int_0^T Yd^-B$ via the operator $\mathcal{K}Y(T,\cdot)$ under regularity conditions on $(Y,\varphi^{(1)}_Y, \varphi^{(2)}_Y)$. Starting from (\ref{Kepsilon}), the appearance of the second martingale derivative is due to the following $L^2(\Omega)$-limit
\begin{eqnarray}
\label{limisec}\lim_{\epsilon\downarrow 0}\int_0^T \int_0^t\frac{1}{\epsilon}\Delta_\epsilon K(t,u)\varphi^{(1)}_Y(t,u)dudt &=& \int_0^T \int_0^t \mathbb{E}[\varphi^{(1)}_Y(t,u)]\frac{\partial K}{\partial t}(t,u)dudt\\
\nonumber&+& \int_0^T \int_r^T \int_r^t \varphi^{(2)}_Y(t,v; r) \frac{\partial K}{\partial t}(t,v)dvdtdW_r.
\end{eqnarray}
The limit (\ref{limisec}) can be controlled via $\varphi^{(1)}_Y$ due to (\ref{cond5}) in Lemma \ref{bassa3}. Then, we can state the following corollary of Theorem \ref{Koperator}.

\begin{corollary}\label{CorMTH}
  Suppose that
  $Y$ satisfies the Assumptions \ref{I0}, \ref{I1}, \ref{I3}, \ref{I5} and (\ref{I6}). Then, $\int_0^T Y_sd^-B_s$ follows (\ref{can}).  
\end{corollary}


\subsection{Proof of Theorem \ref{Koperator}}
From (\ref{Kepsilon}), we should investigate the limits

$$\lim_{\epsilon\downarrow 0}\mathbb{E}\int_0^T|\mathcal{K}_\epsilon Y(T,r) - \mathcal{K} Y(T,r)|^2dr=0$$
and 
$$
\lim_{\epsilon \downarrow 0}\int_0^T \int_0^t \mathbb{E}[\varphi^{(1)}_Y(t,u)]\frac{1}{\epsilon}\Delta_\epsilon K(t,u)dudt = \int_0^T \int_0^t \mathbb{E}[\varphi^{(1)}_Y(t,u)]\frac{\partial K}{\partial t} (t,u)dudt,
$$
where

$$
\mathcal{K}_\epsilon Y (T,r) = \int_r^T \varphi^{(1)}_Y(t,r)\mathcal{D} ^\epsilon_{r,t}B dt + \int_r^T \int_r^t \varphi^{(2)}_Y(t,v;r)\frac{1}{\epsilon}\Delta_\epsilon K(t,v)dvdt
$$
$$+ \int_r^T \mathbb{E}^r [Y_t] \frac{1}{\epsilon}\Delta_\epsilon K(t,r)dt.
$$
In the sequel, we denote 

$$\frac{\partial^2 R^{(r)}}{\partial x_1 \partial x_2}(x_1,x_2):= \int_0^{r} \frac{\partial K}{\partial x_1}(x_1,s) \frac{\partial K}{\partial x_2}(x_2,s)ds,$$
for $r \le x_1 \wedge x_2$.

\begin{proposition}\label{firstISO}
  Suppose that the process $Y$
  satisfies Assumptions \ref{I1} and \ref{I3}. Then,

$$
\mathbb{E}\int_0^T \Bigg|\int_r^T \varphi^{(1)}_Y(t,r)\{\mathcal{D}^\epsilon_{r,t}B-\mathcal{D}_{r,t}B\} dt\Bigg|^2dr\rightarrow 0,
$$
as $\epsilon\downarrow 0$. Moreover, 

$$\mathbb{E}\int_0^T \Bigg|\int_r^T \varphi^{(1)}_Y(t,r)\mathcal{D}_{r,t}B dt\Bigg|^2dr=$$
$$\mathbb{E}\int_{[0,T]^2} \int_0^{x_1\wedge x_2}\varphi^{(1)}_Y(x_1,r)\varphi^{(1)}_Y(x_2,r) \frac{\partial^2 R^{(r)}}{\partial x_1 \partial x_2}(x_1,x_2) drdx_1dx_2$$
$$+ \mathbb{E}\int_{[0,T]^2} \int_0^{x_1\wedge x_2}\varphi^{(1)}_Y(x_1,r)\varphi^{(1)}_Y(x_2,r) \mathbf{D}_{x_1,x_2,r}B drdx_1dx_2$$
$$+ \mathbb{E}\int_{[0,T]^2}\int_0^{x_1\wedge x_2}\varphi^{(1)}_Y(x_1,r)\varphi^{(1)}_Y(x_2,r) \mathbf{D}_{x_2,x_1,r}B drdx_1dx_2.$$

\end{proposition}

We now devote our attention to the proof of Proposition \ref{firstISO}. Our strategy is the following. We will check the sequence $\int_\cdot^{T} \varphi^{(1)}_Y(t,\cdot) \mathcal{D}^\epsilon_{\cdot, t}Bdt$ is weakly relatively compact in $L^2_a(\Omega\times [0,T])$ and it is converging weakly in $L^1(\Omega\times [0,T])$ to the process $\int_\cdot^{T} \varphi^{(1)}_Y(t,\cdot) \mathcal{D}_{\cdot, t}Bdt \in L^2(\Omega\times [0,T])$. This will show that the entire sequence 

$$\int_\cdot^{T} \varphi^{(1)}_Y(t,\cdot) \mathcal{D}^\epsilon_{\cdot, t}Bdt\rightarrow \int_\cdot^{T} \varphi^{(1)}_Y(t,\cdot)\mathcal{D}_{\cdot, t}Bdt$$
converges weakly in $L^2_a(\Omega\times [0,T])$ as $\epsilon\downarrow 0$. At the end, we aim to check  

$$
\mathbb{E}\int_0^T \Bigg|\int_r^T \varphi^{(1)}_Y(t,r)\mathcal{D}^\epsilon_{r,t}B dt\Bigg|^2dr\rightarrow \mathbb{E}\int_0^T \Bigg|\int_r^T \varphi^{(1)}_Y(t,r)\mathcal{D}_{r,t}B dt\Bigg|^2dr,
$$
as $\epsilon\downarrow 0$.

\begin{lemma}\label{L1compac}
Suppose that $Y$ satisfies the Assumption \ref{I1}. Then,  

$$
\mathbb{E}\int_0^T \Bigg\{\int_r^T \varphi^{(1)}_Y(t,r)\{\mathcal{D}^\epsilon_{r,t}B-\mathcal{D}_{r,t}B\} dt\Bigg\} Z_r dr\rightarrow 0, 
$$
for every $Z \in L^\infty_a$. 
\end{lemma}
\begin{proof}
  Consider an arbitrary $Z \in L^\infty_a$. Cauchy-Schwarz's inequality and Assumption
  \ref{I1} yield 

$$\Bigg|\mathbb{E}\int_0^T \Bigg\{\int_r^T \varphi^{(1)}_Y(t,r)\{\mathcal{D}^\epsilon_{r,t}B-\mathcal{D}_{r,t}B\} dt\Bigg\} Z_r dr\Bigg|$$ 
$$= \Bigg|\mathbb{E}\int_0^T \Bigg\{\int_r^T \varphi^{(1)}_Y(t,r)Z_r\{\mathcal{D}^\epsilon_{r,t}B-\mathcal{D}_{r,t}B\} dt\Bigg\} dr\Bigg| $$
$$\le\mathbb{E}\int_0^T \Bigg(\int_r^T |\varphi^{(1)}_Y(t,r)Z_r|^2dt\Bigg)^{\frac{1}{2}} \Bigg(\int_r^T|\mathcal{D}^\epsilon_{r,t}B-\mathcal{D}_{r,t}B|^2 dt\Bigg)^{\frac{1}{2}} dr $$
$$\le \| Z\|_\infty \int_0^T \Bigg(\mathbb{E}\int_r^T |\varphi^{(1)}_Y(t,r)|^2dt\Bigg)^{\frac{1}{2}} \Bigg(\mathbb{E}\int_r^T|\mathcal{D}^\epsilon_{r,t}B-\mathcal{D}_{r,t}B|^2 dt\Bigg)^{\frac{1}{2}} dr$$
$$\le  \| Z\|_\infty \Bigg\{ \int_0^T \mathbb{E}\int_r^T |\varphi^{(1)}_Y(t,r)|^2dtdr\Bigg\}^{\frac{1}{2}} \Bigg\{\int_0^T \mathbb{E}\int_r^T|\mathcal{D}^\epsilon_{r,t}B-\mathcal{D}_{r,t}B|^2 dt dr\Bigg\}^{\frac{1}{2}}.$$
Observe that Assumption \ref{I1}, It\^o's isometry and Fubini's theorem imply 
$$\int_0^T \mathbb{E}\int_r^T |\varphi^{(1)}_Y(t,r)|^2dtdr < \infty.$$ 
We claim that  

\begin{equation}\label{cl1}
\lim_{\epsilon\downarrow 0}\int_0^T\int_r^T\mathbb{E}|\mathcal{D}^\epsilon_{r,t}B-\mathcal{D}_{r,t}B|^q dtdr=0, 
\end{equation}
for $1\le q < \frac{1}{1-H}$. Burkholder-Davis-Gundy's inequality yields

$$ \mathbb{E}|\mathcal{D}^\epsilon_{r,t}B - \mathcal{D}_{r,t}B|^{q}\lesssim_{q,T}\Bigg(\int_0^r \Big| \frac{\partial K}{\partial t}(t,u) - \frac{1}{\epsilon} \Delta_\epsilon K(t,u)  \Big|^2du\Bigg)^{\frac{q}{2}},$$
for $ 0 < r < t < T$. Observe that 

$$\frac{\partial K}{\partial t}(t,u) = c_H (t-u)^{H-\frac{3}{2}}$$
and a simple application of mean value theorem yields 

$$\frac{\partial K}{\partial t}(t,u) - \frac{1}{\epsilon} \Delta_\epsilon K(t,u) = c_H \Big\{(t-u)^{H-\frac{3}{2}} - (\bar{t}_\epsilon -u)^{H-\frac{3}{2}}\Big\},$$
for $u < r < t < \bar{t}_\epsilon < t+\epsilon$ and a constant $c_H$ which depends on $H$. Then, 
$$\Big|\frac{\partial K}{\partial t}(t,u) - \frac{1}{\epsilon} \Delta_\epsilon K(t,u)\Big|^2 \rightarrow 0$$
as $\epsilon \downarrow 0$, for each $0 < u < r < t < T$. We now observe that 

\begin{equation}\label{Depsilonineq1}
\Big|\frac{\partial K}{\partial t}(t,u) - \frac{1}{\epsilon} \Delta_\epsilon K(t,u)\Big|^2\lesssim_H (t-u)^{2H-3}
\end{equation}
for every $\epsilon >0$ and for every $0 < u < r < t < T$. Moreover,

\begin{eqnarray}\label{Depsilonineq2}
\nonumber\int_0^r (t-u)^{2H-3}du &=& \frac{1}{2-2H}\big\{ (t-r)^{2H-2} - t^{2H-2}   \big\}\\
&\le& \frac{1}{2-2H}(t-r)^{2H-2} < \infty,
\end{eqnarray}
for every $0 < r < t < T$. Bounded convergence theorem yields 

$$\mathbb{E}|\mathcal{D}^\epsilon_{r,t}B - \mathcal{D}_{r,t}B|^{q}\lesssim\Bigg(\int_0^r \Big| \frac{\partial K}{\partial t}(t,u) - \frac{1}{\epsilon} \Delta_\epsilon K(t,u)  \Big|^2du\Bigg)^{\frac{q}{2}}\rightarrow 0,$$
for every $r < t < T$ as $\epsilon\downarrow 0$. By using (\ref{Depsilonineq1}) and (\ref{Depsilonineq2}), we can estimate

\begin{equation}\label{Depsilonineq3}
\Bigg(\int_0^r \Big| \frac{\partial K}{\partial t}(t,u) - \frac{1}{\epsilon} \Delta_\epsilon K(t,u)  \Big|^2du\Bigg)^{\frac{q}{2}}\lesssim_H (t-r)^{q (H-1)}
\end{equation}
for every $0 < r < t < T$. If $1\le q < \frac{1}{1-H}$, the right-hand side of (\ref{Depsilonineq3}) is integrable over the set $\{(r,t); 0 < r < t < T\}$. By applying bounded convergence theorem, we conclude (\ref{cl1}).

\end{proof}

\begin{lemma}\label{bofder}
Under (\ref{modHtensor}) and Assumption \ref{I3}, we have 
$$
\sup_{\epsilon \in (0,1)}\mathbb{E}\int_0^T \Bigg|\int_r^T \varphi^{(1)}_Y(t,r)\mathcal{D}^\epsilon_{r,t}B dt\Bigg|^2dr<\infty. 
$$
\end{lemma}

\begin{proof}
Fix $\epsilon \in (0,1)$. By applying Fubini's theorem, we rewrite

\begin{eqnarray*}
\mathbb{E}\int_0^T \Bigg|\int_r^T \varphi^{(1)}_Y(t,r)\mathcal{D}^\epsilon_{r,t}B dt\Bigg|^2dr&=&\mathbb{E}\int_0^T \int_0^T \int_0^{x_1\wedge x_2}\varphi^{(1)}_Y(x_1,r)\varphi^{(1)}_Y(x_2,r)\\
&\times& \mathcal{D}^\epsilon_{r,x_1}B \mathcal{D}^\epsilon_{r,x_2}B drdx_1dx_2.
\end{eqnarray*}

Integration by parts yields 

\begin{eqnarray*}
\mathcal{D}^\epsilon_{r,x_1}B \mathcal{D}^\epsilon_{r,x_2}B&=& \int_0^r \frac{1}{\epsilon}\Delta_\epsilon K(x_1,s) \frac{1}{\epsilon}\Delta_\epsilon K(x_2,s) ds + \int_0^r \big(\mathcal{D}^\epsilon_{s,x_1}B\big) \frac{1}{\epsilon}\Delta_\epsilon K(x_2,s)dW_s\\
&+&  \int_0^r \big(\mathcal{D}^\epsilon_{s,x_2}B \big)\frac{1}{\epsilon}\Delta_\epsilon K(x_1,s)dW_s\\
&=:& A_1(\epsilon; r,x_1,x_2) + A_2(\epsilon; r,x_1,x_2) + A_3(\epsilon; r,x_1,x_2), 
\end{eqnarray*}
for $r < x_1 \wedge x_2$. Then,

$$\mathbb{E}\int_0^T \Bigg|\int_r^T \varphi^{(1)}_Y(t,r)\mathcal{D}^\epsilon_{r,t}B dt\Bigg|^2dr=$$
$$ \sum_{m=1}^3 \mathbb{E}\int_0^T \int_0^T \int_0^{x_1\wedge x_2}\varphi^{(1)}_Y(x_1,r)\varphi^{(1)}_Y(x_2,r) A_m(\epsilon; r,x_1,x_2) drdx_1dx_2.$$

We start by treating $A_1$. A direct application of Mean Value Theorem along the intervals $(x_1, x_1+\epsilon)$ and $(x_2,x_2+\epsilon)$ yields

$$0 < \int_0^r \frac{1}{\epsilon}\Delta_\epsilon K(x_1,s) \frac{1}{\epsilon}\Delta_\epsilon K(x_2,s) ds\le \int_0^{x_1\wedge x_2} \frac{\partial K}{\partial t}(x_1,s) \frac{\partial K}{\partial t}(x_2,s)ds = \frac{\partial^2 R}{\partial x_1\partial x_2}(x_1,x_2) $$ 
for every $s < x_1 \wedge x_2$ and $\epsilon \in (0,1)$. Then

$$\int_0^T \int_0^T \int_0^{x_1\wedge x_2}|\mathbb{E}\varphi^{(1)}_Y(x_1,r)\varphi^{(1)}_Y(x_2,r) A_1(\epsilon;r,x_1,x_2)| drdx_1dx_2$$
$$\le \int_0^T \int_0^T \int_0^{x_1\wedge x_2}\Big|\mathbb{E}\big[\varphi^{(1)}_Y(x_1,r)\varphi^{(1)}_Y(x_2,r)\big] \Big| \frac{\partial^2 R}{\partial x_1\partial x_2}(x_1,x_2) drdx_1dx_2$$
$$ = 2\int_0^T \int_0^{x_2} \int_0^{x_1}\Big|\mathbb{E}\big[\varphi^{(1)}_Y(x_1,r)\varphi^{(1)}_Y(x_2,r)\big] \Big| \frac{\partial^2 R}{\partial x_1\partial x_2}(x_1,x_2)drdx_1dx_2.$$
At this point, we make use Cauchy-Schwarz's inequality twice so that

$$
\Big|\mathbb{E}\big[\varphi^{(1)}_Y(x_1,r)\varphi^{(1)}_Y(x_2,r)\big] \Big|\le \| \varphi^{(1)}_Y(x_1,r)\|_{L^2(\Omega)} \| \varphi^{(1)}_Y(x_2,r)\|_{L^2(\Omega)} 
$$
and 
$$\int_0^T \int_0^{x_2} \int_0^{x_1}\Big|\mathbb{E}\big[\varphi^{(1)}_Y(x_1,r)\varphi^{(1)}_Y(x_2,r)\big] \Big| \frac{\partial^2 R}{\partial x_1\partial x_2}(x_1,x_2)drdx_1dx_2$$
$$\le \int_0^T \int_0^{x_2} \Bigg(\int_0^{x_1}\| \varphi^{(1)}_Y(x_1,r)\|^2_{L^2(\Omega)}dr\Bigg)^{\frac{1}{2}} \Big(\int_0^{x_1}\| \varphi^{(1)}_Y(x_2,r)\|^2_{L^2(\Omega)}dr\Bigg)^{\frac{1}{2}} \frac{\partial^2 R}{\partial x_1\partial x_2}(x_1,x_2)dx_1dx_2.$$
By Burkholder-Davis-Gundy inequality, we know that 
$$
\mathbb{E}|\bar{Y}_{x_1}|^p \sim_p \mathbb{E}\Big( \int_0^{x_1} |\varphi^{(1)}_Y(x_1,s)|^2ds\Big)^{\frac{p}{2}},\quad \mathbb{E}|\mathbb{E}^{x_1}[\bar{Y}_{x_2}]|^p \sim_p \mathbb{E}\Big( \int_0^{x_1} |\varphi^{(1)}_Y(x_2,s)|^2ds\Big)^{\frac{p}{2}},
$$
for $x_1 < x_2$ and $1 < p  < \infty$.  Hence (see Remark \ref{modHtrem}), 

$$\Bigg|\mathbb{E}\int_0^T \int_0^T \int_0^{x_1\wedge x_2}\varphi^{(1)}_Y(x_1,r)\varphi^{(1)}_Y(x_2,r) A_1(\epsilon;r,x_1,x_2) drdx_1dx_2\Bigg|$$

$$\le 2 \int_0^T \int_0^{x_2} \big\| \bar{Y}_{x_1}\big\|_{L^2(\Omega)} \big\| \mathbb{E}^{x_1}[\bar{Y}_{x_2}] \big\|_{L^2(\Omega)} \frac{\partial^2 R}{\partial x_1\partial x_2}(x_1,x_2)dx_1dx_2 $$
$$\le 8 \int_0^T \int_0^T \big\| Y_{x_1}\big\|_{L^2(\Omega)} \big\| Y_{x_2} \big\|_{L^2(\Omega)} \frac{\partial^2 R}{\partial x_1\partial x_2}(x_1,x_2)dx_1dx_2.$$

Observe that $A_2(\epsilon;r,x_1,x_2)=\mathbf{D}^\epsilon_{x_1,x_2;r}B$ and $A_2(\epsilon;r,x_2,x_1) = A_3(\epsilon;r,x_1,x_2) = \mathbf{D}^\epsilon_{x_2,x_1;r}B$. By symmetry, it is enough to treat $A_2$. By applying Lemma \ref{Depsilonbounds}, we conclude the proof. 

\end{proof}

By following the same arguments given in the proof of Lemma \ref{bofder}, the following estimate hold true. 

\begin{lemma}\label{L2norm1}
If (\ref{modHtensor}) and Assumption \ref{I3} are fulfilled with conjugate exponents $1< p \le \infty$ and $1\le q < \infty$, then  
\begin{small}
$$
\mathbb{E}\int_0^T \Bigg|\int_r^T \varphi^{(1)}_Y(t,r)\mathcal{D}_{r,t}B dt\Bigg|^2dr\le  
\int_0^T\int_0^T \big\| \bar{Y}_{x_1}\big\|_{L^2(\Omega)} \big\| \mathbb{E}^{x_1}[\bar{Y}_{x_2}] \big\|_{L^2(\Omega)} \frac{\partial^2 R}{\partial x_1 \partial x_2}(x_1,x_2)dx_1dx_2
$$
$$
+\int_{0}^T \int_0^T \int_0^{x_1\wedge x_2} \|\varphi^{(1)}_Y(x_1,r)\varphi^{(1)}_Y(x_2,r)\|_{L^p(\Omega)} \sup_{\epsilon \in (0,1)}\big\|\mathbf{D}^\epsilon_{x_1,x_2;r}B\big\|_{L^q(\Omega)}drdx_1dx_2
$$

$$ + \int_{0}^T \int_0^T \int_0^{x_1\wedge x_2} \|\varphi^{(1)}_Y(x_1,r)\varphi^{(1)}_Y(x_2,r)\|_{L^p(\Omega)}\sup_{\epsilon \in (0,1)}\big\|\mathbf{D}^\epsilon_{x_2,x_1;r}B\big\|_{L^q(\Omega)} drdx_1dx_2<\infty. 
$$
\end{small}
\end{lemma}

We are now in position to prove Proposition \ref{firstISO}.  

\

\noindent \textbf{Proof of Proposition}~\ref{firstISO}: Lemmas \ref{L1compac}, \ref{bofder} and \ref{L2norm1} yield 

$$\int_\cdot^{T} \varphi^{(1)}_Y(t,\cdot) \mathcal{D}^\epsilon_{\cdot, t}Bdt\rightarrow \int_\cdot^{T} \varphi^{(1)}_Y(t,\cdot)\mathcal{D}_{\cdot, t}Bdt$$
converges weakly in $L^2_a(\Omega\times [0,T])$. We claim  

$$
\mathbb{E}\int_0^T \Bigg|\int_r^T \varphi^{(1)}_Y(t,r)\mathcal{D}^\epsilon_{r,t}B dt\Bigg|^2dr\rightarrow \mathbb{E}\int_0^T \Bigg|\int_r^T \varphi^{(1)}_Y(t,r)\mathcal{D}_{r,t}B dt\Bigg|^2dr,
$$
as $\epsilon\downarrow 0$. From the proof of Lemma \ref{bofder}, we recall we can write

\begin{eqnarray*}
\mathbb{E}\int_0^T \Bigg|\int_r^T \varphi^{(1)}_Y(t,r)\mathcal{D}^\epsilon_{r,t}B dt\Bigg|^2dr&=&\mathbb{E}\int_0^T \int_0^T \int_0^{x_1\wedge x_2}\varphi^{(1)}_Y(x_1,r)\varphi^{(1)}_Y(x_2,r)\\
&\times& \mathcal{D}^\epsilon_{r,x_1}B \mathcal{D}^\epsilon_{r,x_2}B drdx_1dx_2. 
\end{eqnarray*}

For each $r < x_1\wedge x_2$ and $1\le q < \infty$, triangle inequality yields 

$$\|\mathcal{D}^\epsilon_{r,x_1}B \mathcal{D}^\epsilon_{r,x_2}B - \mathcal{D}_{r,x_1}B \mathcal{D}_{r,x_2}B\|_{L^q(\Omega)}$$
$$\le \|\mathcal{D}^\epsilon_{r,x_1}B [ \mathcal{D}^\epsilon_{r,x_2}B - \mathcal{D}_{r,x_2}B]\|_{L^q(\Omega)} $$
$$+ \|\mathcal{D}_{r,x_2}B [\mathcal{D}^\epsilon_{r,x_1}B - \mathcal{D}_{r,x_1}B]\|_{L^q(\Omega)}.$$

Burkholder-Davis-Gundy's inequality yields

$$ \mathbb{E}|\mathcal{D}^\epsilon_{r,y}B - \mathcal{D}_{r,y}B|^{\alpha}\lesssim_{\alpha,T}\Bigg(\int_0^r \Big| \frac{\partial K}{\partial t}(y,u) - \frac{1}{\epsilon} \Delta_\epsilon K(y,u)  \Big|^2du\Bigg)^{\frac{\alpha}{2}}\rightarrow 0,$$
as $\epsilon\downarrow 0$ for any $1 < \alpha < \infty$, for each $y=x_1$ or $y=x_2$ and $r < x_1\wedge x_2$. This shows that 

$$\|\mathcal{D}^\epsilon_{r,x_1}B \mathcal{D}^\epsilon_{r,x_2}B - \mathcal{D}_{r,x_1}B \mathcal{D}_{r,x_2}B\|_{L^q(\Omega)}\rightarrow 0,$$
for each $r < x_1\wedge x_2$ and $1\le q < \infty$. Then, we choose the conjugate exponents $1 < p \le \infty, 1\le q < \infty$ from Assumption \ref{I3} to infer the pointwise convergence

$$\Big|\mathbb{E}\Big[ \varphi^{(1)}_Y(x_1,r)\varphi^{(1)}_Y(x_2,r)\{\mathcal{D}^\epsilon_{r,x_1}B \mathcal{D}^\epsilon_{r,x_2}B  - \mathcal{D}_{r,x_1}B \mathcal{D}_{r,x_2}B\}\Big]\Big|$$
$$\le \big\|\varphi^{(1)}_Y(x_1,r)\varphi^{(1)}_Y(x_2,r)\|_{L^p(\Omega)} \big\|\mathcal{D}^\epsilon_{r,x_1}B \mathcal{D}^\epsilon_{r,x_2}B - \mathcal{D}_{r,x_1}B \mathcal{D}_{r,x_2}B\big\|_{L^q(\Omega)}\rightarrow 0,$$
as $\epsilon \downarrow 0$, for each $r < x_1 \wedge x_2$. From the proof of Lemma \ref{bofder}, we know there exists $f \in L^1(G)$ such that

$$\Big|\mathbb{E}\big[ \varphi^{(1)}_Y(x_1,r)\varphi^{(1)}_Y(x_2,r)\mathcal{D}^\epsilon_{r,x_1}B \mathcal{D}^\epsilon_{r,x_2}B\big]\Big|\lesssim_{q,p,H,T}f (r,x_1,x_2)$$
for every $\epsilon \in (0,1)$ and for each $(r,x_1,x_2) \in G=\{(a,b,c); a < b\wedge c\}$. The function $f$ is given by 

$$f(r,x_1,x_2) = \big\| \bar{Y}_{x_1}\big\|_{L^2(\Omega)} \big\| \mathbb{E}^{x_1}[\bar{Y}_{x_2}] \big\|_{L^2(\Omega)} \frac{\partial^2 R}{\partial x_1\partial x_2}(x_1,x_2)$$
$$+ \|\varphi^{(1)}_Y(x_1,r)\varphi^{(1)}_Y(x_2,r)\|_{L^p(\Omega)} \sup_{\epsilon \in (0,1)}\big\|\mathbf{D}^\epsilon_{x_1,x_2;r}B\big\|_{L^q(\Omega)}$$
$$+  \|\varphi^{(1)}_Y(x_1,r)\varphi^{(1)}_Y(x_2,r)\|_{L^p(\Omega)} \sup_{\epsilon \in (0,1)}\big\|\mathbf{D}^\epsilon_{x_2,x_1;r}B\big\|_{L^q(\Omega)},$$
for $r < x_1\wedge x_2$. Bounded convergence theorem yields

$$\lim_{\epsilon\downarrow 0}\mathbb{E}\int_0^T \Bigg|\int_r^T \varphi^{(1)}_Y(t,r)\mathcal{D}^\epsilon_{r,t}B dt\Bigg|^2dr= \mathbb{E}\int_0^T \Bigg|\int_r^T \varphi^{(1)}_Y(t,r)\mathcal{D}_{r,t}B dt\Bigg|^2dr$$
$$=\mathbb{E}\int_{[0,T]^2} \int_0^{x_1\wedge x_2}\varphi^{(1)}_Y(x_1,r)\varphi^{(1)}_Y(x_2,r) \frac{\partial^2 R^{(r)}}{\partial x_1 \partial x_2}(x_1,x_2) drdx_1dx_2$$
$$+ \mathbb{E}\int_{[0,T]^2} \int_0^{x_1\wedge x_2}\varphi^{(1)}_Y(x_1,r)\varphi^{(1)}_Y(x_2,r) \mathbf{D}_{x_1,x_2,r}B drdx_1dx_2$$
$$+ \mathbb{E}\int_{[0,T]^2}\int_0^{x_1\wedge x_2}\varphi^{(1)}_Y(x_1,r)\varphi^{(1)}_Y(x_2,r) \mathbf{D}_{x_2,x_1,r}B drdx_1dx_2.$$
This concludes the proof of Proposition \ref{firstISO}.

\begin{lemma}\label{secconve}
Suppose that a square-integrable process $Y$ satisfies Assumption \ref{I4}. Then,  

$$\lim_{\epsilon\downarrow 0}\mathbb{E}\int_0^T \Bigg|\int_r^T \int_r^t \varphi^{(2)}_Y(t,v;r)\Big\{\frac{1}{\epsilon}\Delta_\epsilon K(t,v) - \frac{\partial K}{\partial t}(t,v)\Big\}
dvdt\Bigg|^2dr=0.
$$
\end{lemma}
\begin{proof}
Of course, we have pointwise convergence 
$$\lim_{\epsilon \downarrow 0}\varphi^{(2)}_Y(t,v;r)\Big\{\frac{1}{\epsilon}\Delta_\epsilon K(t,v) - \frac{\partial K}{\partial t}(t,v)\Big\}=0$$
almost surely, for each $r < v < t$. Moreover, a simple application of Mean Value theorem yields

$$\Big|\varphi^{(2)}_Y(t,v;r)\Big\{\frac{1}{\epsilon}\Delta_\epsilon K(t,v) - \frac{\partial K}{\partial t}(t,v)\Big\}\Big|\le 2 |\varphi^{(2)}_Y(t,v;r)|\frac{\partial K}{\partial t}(t,v),$$
for each $r < v < t$. By Assumption \ref{I4},   

$$\mathbb{E}\Bigg|\int_r^T \int_r^t \varphi^{(2)}_Y(t,v;r)\Big\{\frac{1}{\epsilon}\Delta_\epsilon K(t,v) - \frac{\partial K}{\partial t}(t,v)\Big\}
dvdt\Bigg|^2$$
$$\le 4\mathbb{E}\Bigg|\int_r^T \int_r^t |\varphi^{(2)}_Y(t,v;r)|\frac{\partial K}{\partial t}(t,v)dvdt\Bigg|^2 < \infty,$$
for almost all $r$ w.r.t. Lebesgue on $[0,T]$. Hence,  

$$\lim_{\epsilon\downarrow 0}\Bigg(\int_r^T \int_r^t \varphi^{(2)}_Y(t,v;r)\Big\{\frac{1}{\epsilon}\Delta_\epsilon K(t,v) - \frac{\partial K}{\partial t}(t,v)\Big\}
dvdt\Bigg)^2=0$$
almost surely and for almost all $r$. Moreover, 

$$\Bigg(\int_r^T \int_r^t \varphi^{(2)}_Y(t,v;r)\Big\{\frac{1}{\epsilon}\Delta_\epsilon K(t,v) - \frac{\partial K}{\partial t}(t,v)\Big\}
dvdt\Bigg)^2$$
\begin{equation}\label{secderla}
\le 4 \Bigg|\int_r^T \int_r^t |\varphi^{(2)}_Y(t,v;r)|\frac{\partial K}{\partial t}(t,v)dvdt\Bigg|^2,
\end{equation}
where Assumption \ref{I4} ensures that the right-hand side of (\ref{secderla}) is integrable w.r.t. $\mathbb{P}\times \text{Leb}$. Then, bounded convergence theorem allows us to conclude the proof.  
\end{proof}

\begin{remark}\label{remA4}
A direct application of Fubini's theorem, Jensen and Cauchy-Schwarz inequalities yield the following estimate 

$$\mathbb{E}\int_0^T\Bigg(\int_r^T  |\mathbb{E}^r[Y_t]|\frac{\partial K}{\partial t}(t,r)dt\Bigg)^2dr$$
$$=\int_{[0,T]^2}\int_0^{x_1\wedge x_2} \mathbb{E} \big|\mathbb{E}^r[Y_{x_1}] \mathbb{E}^r[Y_{x_2}]\big| \frac{\partial K}{\partial x_1}(x_1,r)\frac{\partial K}{\partial x_2}(x_2,r)drdx_1dx_2 $$
$$\le \int_{[0,T]^2} \big\|Y_{x_1}\|_{L^2(\Omega)} \big\|Y_{x_2}\|_{L^2(\Omega)} \frac{\partial^2 R}{\partial x_1\partial x_2}(x_1,x_2)dx_1dx_2.$$
Similarly,

$$\int_0^T\Bigg(\int_r^T  |\mathbb{E}[Y_t]|\frac{\partial K}{\partial t}(t,r)dt\Bigg)^2dr$$
$$=\int_{[0,T]^2}\int_0^{x_1\wedge x_2}  \big|\mathbb{E}[Y_{x_1}] \mathbb{E}[Y_{x_2}]\big| \frac{\partial K}{\partial x_1}(x_1,r)\frac{\partial K}{\partial x_2}(x_2,r)drdx_1dx_2 $$
$$= \int_{[0,T]^2} \big|\mathbb{E}[Y_{x_1}] \mathbb{E}[Y_{x_2}]\big| \frac{\partial^2 R}{\partial x_1\partial x_2}(x_1,x_2)dx_1dx_2.$$

\end{remark} 
\begin{lemma}\label{restconv}
If (\ref{modHtensor}) and Assumption \ref{I5} are fulfilled, then   
\begin{equation}\label{cond-1}
\lim_{\epsilon\downarrow 0}\mathbb{E}\int_0^T \Bigg| \int_r^T \mathbb{E}^r [Y_t] \Big(\frac{1}{\epsilon}\Delta_\epsilon K(t,r)-\frac{\partial K}{\partial t}(t,r)\Big)dt\Bigg|^2dr=0
\end{equation}
and

\begin{equation}\label{cond-2}
\lim_{\epsilon \downarrow 0}\int_0^T \int_0^t \mathbb{E}[\varphi^{(1)}_Y(t,u)]\frac{1}{\epsilon}\Delta_\epsilon K(t,u)dudt = \int_0^T \int_0^t \mathbb{E}[\varphi^{(1)}_Y(t,u)]\frac{\partial K}{\partial t} (t,u)dudt.
\end{equation}
\end{lemma}
\begin{proof}
Under Assumption \ref{I5}, a simple application of bounded convergence theorem allows us to state that (\ref{cond-2}) holds true. The proof of (\ref{cond-1}) is entirely similar to the proof of Lemma \ref{secconve} (see also Remark \ref{remA4}), so we omit the details.  
\end{proof}
Summing Proposition \ref{firstISO} and Lemmas \ref{secconve} and \ref{restconv}, we conclude the proof of Theorem \ref{Koperator}. In the next sections, we will illustrate Theorem \ref{Koperator} with some examples.

\section{State-dependent case}\label{S3}

The following decomposition will play a key role. We shall represent 

$$B_t = \mathbb{E}^s [B_t] + B_t - \mathbb{E}^s 
[B_t],$$
where
$$\mathbb{E}^s[B_t]  = \int_{0}^s K(t,u)dW_u,\quad B_t - \mathbb{E}^s [B_t] = \int_s^t K(t,u)dW_u,$$
for $0 < s < t$ and $t\ge 0$. In the sequel, we denote 

\begin{eqnarray*}
\theta(s,t)&:=& \text{Var} \big(B_t - \mathbb{E}^s [B_t]\big)\\
&=& \frac{1}{2H}(t-s)^{2H},
\end{eqnarray*} 
for $0 < s < t\le T$.

For the sake of preciseness, we first recall the following property.
We say that a function $f:[0,T]\times \mathbb{R}\rightarrow\mathbb{R}$ has polynomial growth if there exists $q>0$ such that 

$$|f(t,x)|\lesssim_{q,T} (1 + |x|^q),$$
for every $t \in (0,T]$ and $x \in \mathbb{R}$.   

\begin{lemma}\label{repstate}
If $f:\mathbb{R}_+\times \mathbb{R}\rightarrow \mathbb{R}$ has polynomial growth, then 
$$\mathbb{E}^s\big[ f(t,B_t)\big] = \big(P_{\theta(s,t)}f\big) \big(t, \mathbb{E}^s[B_t]\big)$$
for $0 < s < t$. Here, 
\begin{equation}\label{orinSG}
\big(P_r f\big)(t,x)= \int_{\mathbb{R}} f(t, x+y) \frac{1}{\sqrt{2\pi r}}e^{\frac{-y^2}{2r}}dy = \int_{\mathbb{R}} f(t, z) \frac{1}{\sqrt{2\pi r}}e^{\frac{-(z-x)^2}{2r}}dz,
\end{equation}
for $r>0$ and $x \in \mathbb{R}$.
\end{lemma}
\begin{proof}

To keep notation simple, we denote 

$$\tilde{B}^s_t:= B_t - \mathbb{E}^s [B_t], $$
for $0 < s < t$ and $t\ge 0$. Fix $0 < s < t$ with $0 < t\le T$. Observe 

\begin{eqnarray*}
\mathbb{E}^s \big[f(t,B_t)\big]&=& \mathbb{E}^s \big[f(t,\mathbb{E}^s[B_t] + \tilde{B}^s_t)\big]\\
&=& \mathbb{E} \Big[f(t,\mathbb{E}^s[B_t] + \tilde{B}^s_t)\big| \mathbb{E}^s[B_t]\Big].
\end{eqnarray*} 
Since $\tilde{B}^s_t$ is independent of $\mathbb{E}^s[B_t]$, we then have 

$$\mathbb{E}^s ~\big[f(t,B_t)\big] = \mathbb{E}\Big[ f(t,x + \tilde{B}^s_t)  \Big],\quad \text{for}~x = \mathbb{E}^s[B_t].$$
We know that $\tilde{B}^s_t$ is a zero mean Gaussian random variable with variance $\theta(s,t)$. Then, 

$$\mathbb{E}\Big[ f(t,x + \tilde{B}^s_t)  \Big] = \big(P_{\theta(s,t)} f\big)(t, x).$$

\end{proof}

\begin{lemma}\label{firstderSTATE}
Let $Y_t = g(t,B_t)$, where $g$ has polynomial growth. Then, the martingale derivative $\varphi^{(1)}_Y$ is given by

$$\varphi^{(1)}_Y(t,r) = \partial_x \big(P_{\theta(r,t)} g) (t,\mathbb{E}^r[B_t])K(t,r),$$
for $0 < r < t \le T$. 
\end{lemma}
\begin{proof}
The argument of the proof lies on the properties of the heat semigroup. Fix the time variable $t \in (0,T]$. Let us define 

\begin{eqnarray*}
u(r,x)&:=& \big(P_{\theta(r,t)}g\big)(t,x);~0 < r < t. 
\end{eqnarray*}
From Lemma \ref{repstate}, we know that  $u(r,\mathbb{E}^r[B_t]) = \mathbb{E}^r [g(t,B_t)],$ for $0 < r < t$. 
Derivation chhain rule yields 

\begin{eqnarray}
\nonumber\partial_r u(r,x) &=& \partial_r \big(P_{\theta(r,t)}\big)g(t,x)\\
\label{heatarg1}&=& \partial_s \big(P_sg\big)(t,x)(-1)K^2(t,r);~s=\theta(r,t).
\end{eqnarray}
Recall that $(s,x)\mapsto (P_s g)(t,x)$ is the heat semigroup so that 

\begin{equation}\label{heatarg2}
\partial_s (P_s g)(t,x) = \frac{1}{2} \partial^2_x (P_s g)(t,x).
\end{equation}
for $s>0$ and $x \in \mathbb{R}$. From (\ref{heatarg1}) and (\ref{heatarg2}), we have

\begin{equation}\label{heatarg3}
\partial_r u(r,x) = -\frac{K^2(t,r)}{2} \partial^2_x (P_{\theta(r,t)} g)(t,x) =  -\frac{K^2(t,r)}{2} \partial^2_x u(r,x).
\end{equation}
Observe that $(r,x) \mapsto u(r,x)$ is $C^{1,2}$
so that by applying It\^o's formula and (\ref{heatarg3}), we get 
\begin{eqnarray}
\nonumber u(r,\mathbb{E}^r[B_t])&=& \mathbb{E}[g(t,B_t)] +\int_0^r \partial_s  u(s,\mathbb{E}^s[B_t])ds + \int_0^r \partial_x u (s,\mathbb{E}^s[B_t])K(t,s)dW_s\\
\nonumber&+& \frac{1}{2} \int_0^r \partial^2_x u (s,\mathbb{E}^s[B_t])K^2(t,s)ds\\
\label{heatarg4}&=& \mathbb{E}[g(t,B_t)] + \int_0^r \partial_x u (s,\mathbb{E}^s[B_t])K(t,s)dW_s, 
\end{eqnarray}
for every $r < t$. Observe that $\partial_x u (s,\mathbb{E}^s[B_t])K(t,s) = \partial_x \big(P_{\theta(s,t)} g) (t,\mathbb{E}^s[B_t])K(t,s)$ for $s< t$. By continuity, Lemma \ref{repstate} and (\ref{heatarg4}), we get the local-martingale representation 
\begin{equation}\label{heatarg5}
g(t,B_t) = \mathbb{E}[g(t,B_t)] + \int_0^t \partial_x \big(P_{\theta(s,t)} g) (t,\mathbb{E}^s[B_t])K(t,s)dW_s
\end{equation}
From (\ref{heatarg5}) and the polynomial growth property of $g$, we get 
\begin{equation}\label{heatarg6}
\mathbb{E}\Bigg|\int_0^t \partial_x \big(P_{\theta(s,t)} g) (t,\mathbb{E}^s[B_t])K(t,s)dW_s\Bigg|^2 = \mathbb{E}\big| g(t,B_t)  - g(0,0) \big|^2 < \infty. 
\end{equation}
From (\ref{heatarg5}) and (\ref{heatarg6}), we conclude the proof. 
\end{proof}

\begin{remark}
Let $Y_t = g(t,B_t)$, where $g$ has polynomial growth. Then, $Y$ satisfies Assumption \ref{I1}. 
\end{remark}

In order to check that $Y$ satisfies Assumptions \ref{I3} and \ref{I4}, we need to estimate $\varphi^{(1)}_Y(t,r)$ along the simplex $\{0 < r < t < T\}$. Unfortunately, the previous It\^o-type-formula argument jointly with the heat PDE used in the proof of Lemma \ref{firstderSTATE} do not apply. For this purpose, we apply a direct argument on the heat semigroup and this produces a sufficient regularity condition which fulfills the Assumptions \ref{I3} and \ref{I4} of Theorem \ref{Koperator}. We start with a technical lemma.

\begin{lemma}\label{bfder}
If $g$ has polynomial growth, then there exists an exponent $p>0$ such that

\begin{equation}
\sup_{0\le t\le T, 0 < b\le T}\Bigg|\int_{\mathbb{R}} g(t,z)\frac{(z-x)}{b} \exp \Big(\frac{-(z-x)^2}{2b} \Big) dz \Bigg|\lesssim_{M,p} (1 + |x|^p),
\end{equation}
for every $x \in \mathbb{R}$. 
\end{lemma}
\begin{proof}
By assumption, there exists $p>0$ and $M>$ such that 

$$|g(t,z)|\le M(1 + |z|^p),$$
for every $(t,z) \in [0,T]\times \mathbb{R}$. Make the change of variables $y=z-x$. Then,

$$I(x,b,t) = \int_{\mathbb{R}} g(t,x+y)\frac{y}{b} \exp \Big(\frac{-y^2}{2b} \Big) dy.$$
Then,

$$|I(x,b,t)|\lesssim_{M,q} \int_{\mathbb{R}} \big(1 +|x|^p + |y|^p\big) \frac{|y|}{b} \exp \Big(\frac{-y^2}{2b} \Big) dy $$
$$ = \int_{\mathbb{R}} \frac{|y|}{b} \exp \Big(\frac{-y^2}{2b} \Big) dy + |x|^p\int_{\mathbb{R}} \frac{|y|}{b} \exp \Big(\frac{-y^2}{2b} \Big) dy $$
$$+\int_{\mathbb{R}} \frac{|y|^{p+1}}{b} \exp \Big(\frac{-y^2}{2b} \Big) dy. $$ 
Change the variables $u = \frac{y}{\sqrt{b}}$ and observe that

$$\int_{\mathbb{R}} \frac{|y|}{b} \exp \Big(\frac{-y^2}{2b} \Big) dy = \int_{\mathbb{R}} \frac{|u|\sqrt{b}}{b} \exp \Big(\frac{-u^2}{2}\Big)\sqrt{b} du$$
$$ = \int_{\mathbb{R}} |u| \exp \Big(\frac{-u^2}{2}\Big) du = 2$$
Moreover, 

$$\int_{\mathbb{R}} \frac{|y|^{p+1}}{b} \exp \Big(\frac{-y^2}{2b} \Big) dy =b^{\frac{p}{2}}\int_{\mathbb{R}}|u|^{p+1} \exp \Big(\frac{-u^2}{2}\Big)du$$
$$\le T^{\frac{p}{2}}\int_{\mathbb{R}}|u|^{p+1} \exp \Big(\frac{-u^2}{2}\Big)du < \infty.$$
Then, 
 
$$|I(x,b,t)|_{M,p}\lesssim_{M,p} 2 + |x|^p + T^{\frac{p}{2}}\int_{\mathbb{R}}|u|^{p+1} \exp \Big(\frac{-u^2}{2}\Big)du.$$ 
 This concludes the proof.
\end{proof}

\begin{lemma}\label{firstderBOUND}
If $Y_t=g(t, B_t)$ and $g$ has polynomial growth, then there exists a positive random variable $G$ with moments of all orders such that 

\begin{equation}\label{est1derstate}
|\varphi^{(1)}_Y(t,u)|\lesssim_{T} (t-u)^{-\frac{1}{2}} G
\end{equation}
a.s. for every $0 < u < t \le T$. Therefore, $\varphi^{(1)}_Y$ fulfills Assumption \ref{I3}. 

\end{lemma}
\begin{proof}
Recall 

$$\varphi^{(1)}_Y(t,u)= \partial_x \big(P_{\theta(u,t)} g) (t,\mathbb{E}^u[B_t])K(t,u)$$
for $u < t$, where 

$$\big(P_{\theta(u,t)} g) (t,x)=\frac{1}{\sqrt{2\pi\theta(u,t)}}\int_\mathbb{R} g(t,z)\exp \Big( -\frac{(z-x)^2}{2\theta(u,t)} \Big)dz.$$
Then, 

$$\partial_x \big(P_{\theta(u,t)} g) (t,x)=\frac{1}{\sqrt{2\pi\theta(u,t)}}\frac{1}{\theta(u,t)}\int_\mathbb{R} g(t,z)(z-x)\exp \Big( -\frac{(z-x)^2}{2\theta(u,t)} \Big)dz.$$
By applying Lemma \ref{bfder} and using the polynomial growth of $g$, we get 

$$|\varphi^{(1)}_Y(t,u)|\lesssim_p (t-u)^{-\frac{1}{2}} \big(1+ \sup_{0 \le r \le t\le T} |\mathbb{E}^r [B_t]|^p\big)$$
for $u < t$ and some $p>0$. By applying Fernique's theorem on the Gaussian field $(r,t)\mapsto \mathbb{E}^r [B_t]$ over $\{0 \le r \le t\le T\}$, we conclude the proof of the estimate (\ref{est1derstate}).  Now, for each $r < x_1\wedge x_2$ with $x_1\neq x_2$, we have

$$\|\varphi^{(1)}_Y(x_1,r) \varphi^{(1)}_Y(x_2,r)\|_{L^q(\Omega)}\lesssim_{T,H,q} (x_1-r)^{-\frac{1}{2}} (x_2-r)^{-\frac{1}{2}},$$
for every $1\le q < \infty$. Choose for instance $q=2$ and observe that (see Remark \ref{supq2})

$$\int_0^T \int_0^{x_2}\int_0^{x_1} (x_1-r)^{-\frac{1}{2}} (x_2-r)^{-\frac{1}{2}} \Bigg(\int_0^r (x_1-s)^{2(H-1)}(x_2-s)^{2(H-\frac{3}{2})}ds\Bigg)^{\frac{1}{2}}dr dx_1dx_2$$
$$\le \int_{0}^T\int_0^{x_2} \int_0^{x_1} (x_1-r)^{H-\frac{3}{2}} (x_2-x_1)^{-\frac{1}{2}} \Bigg(\int_0^r (x_2-s)^{2(H-\frac{3}{2})}ds\Bigg)^{\frac{1}{2}}dr dx_1dx_2 $$
$$=\int_{0}^T\int_0^{x_2} \int_0^{x_1} (x_1-r)^{H-\frac{3}{2}} (x_2-x_1)^{-\frac{1}{2}} (x_2-r)^{H-1} dr dx_1dx_2$$
\begin{equation}\label{prodhalf}
\le \int_{0}^T\int_0^{x_2} \int_0^{x_1} (x_1-r)^{H-\frac{3}{2}} (x_2-x_1)^{-\frac{1}{2}} (x_2-x_1)^{H-1} dr dx_1dx_2 < \infty.
\end{equation}
A similar argument also allows us to state 

$$\int_0^T \int_0^{x_1}\int_0^{x_2} (x_1-r)^{-\frac{1}{2}} (x_2-r)^{-\frac{1}{2}} \Bigg(\int_0^r (x_1-s)^{2(H-1)}(x_2-s)^{2(H-\frac{3}{2})}ds\Bigg)^{\frac{1}{2}}dr dx_2dx_1 < \infty.$$

\end{proof}

Next, we investigate Assumption \ref{I4} related to the second martingale derivative  $\varphi_Y^{(2)}(t,v; r)$ for $r < v < t $. Instead of checking Assumption \ref{I6} and remain on the level of $\varphi^{(1)}_Y$, we go to the second derivative $\varphi^{(2)}_Y$ because we need to compute it anyway to provide the most explicit expression for $\int g(B)d^-B$. For this purpose, let us fix $v < t$. From Lemma \ref{firstderSTATE}, we know that 

$$\varphi^{(1)}_Y(t,v) = \partial_x \big(P_{\theta(v,t)} g) (t,\mathbb{E}^v[B_t])K(t,v).$$
Next, in order to compute explicitly $\varphi^{(2)}_Y$, we will make use of Malliavin calculus on the Gaussian space of the underlying Brownian motion $W$. More precisely, Clark-Ocone formula associated with $W$. Let $\mathbb{D}^{1,2}$ be the Sobolev-Malliavin space associated with $W$. 
We claim $\partial_x \big(P_{\theta(v,t)} g) (t,\mathbb{E}^v[B_t]) \in \mathbb{D}^{1,2}$. In fact, let us denote 

$$p_s(z):=\frac{1}{\sqrt{2\pi s}} \exp \big(  \frac{-z^2}{2s} \big)$$
Then, 

$$(P_sg)(t,x) = \int_\mathbb{R}g(t,y)p_s(y-x)dy$$
so that 

$$\partial_x (P_sg)(t,x) = \int_\mathbb{R}g(t,y)p_s(y-x)\frac{(y-x)}{s}dy$$
and

\begin{equation}\label{dersecder}
\partial^2_x \big(P_sg\big)(t,x) = \int_\mathbb{R}g(t,y)\Big[ \frac{(y-x)^2}{s^2} - \frac{1}{s} \Big]p_s(y-x)dy. 
\end{equation}

\begin{lemma}\label{growthsec}
If $g$ has polynomial growth, then there exists $q>0$ such that 
$$|\partial^2_x P_sg(t,x)|\lesssim_q s^{-1} \big( 1 + |x|^q\big),$$
for every $t,s \in (0,T]$ and $x \in \mathbb{R}$. 
\end{lemma}
\begin{proof}
From (\ref{dersecder}), we have 

$$|\partial^2_x \big(P_sg\big)(t,x)| \le \int_\mathbb{R}|g(t,y)|\Big[ \frac{(y-x)^2}{s^2} + \frac{1}{s} \Big]p_s(y-x)dy. 
$$
Change the variables $z=y-x$. Then, 

$$|\partial^2_x \big(P_sg\big)(t,x)| \le \int_\mathbb{R}|g(t,z+x)|\Big[ \frac{z^2}{s^2} + \frac{1}{s} \Big]p_s(z)dz. 
$$
where $|g(t,z+x)|\lesssim (1 + |x|^q + |z|^q)$, for every $(t,z,x) \in [0,T]\times \mathbb{R}^2$ for some $q>0$. Then, 

$$|\partial^2_x \big(P_sg\big)(t,x)| \lesssim \frac{1}{\sqrt{2\pi s}} \int_\mathbb{R}[1 + |x|^q + |z|^q]\Big[ \frac{z^2}{s^2} + \frac{1}{s} \Big]e^{\frac{-z^2}{2s}}dz. 
$$
Make the change of variable $u = \frac{z}{\sqrt{s}}$ so that $e^{\frac{-z^2}{2s}} = e^{\frac{-u^2}{2}}$ and $\frac{z^2}{s^2} = \frac{u^2}{s}$. As a result, 

$$|\partial^2_x \big(P_sg\big)(t,x)| \lesssim \frac{1}{\sqrt{2\pi s}} \int_\mathbb{R}[1 + |x|^q + s^{\frac{q}{2}}|u|^q]\Big[ \frac{u^2}{s} + \frac{1}{s} \Big]e^{\frac{-u^2}{2}}\sqrt{s}du 
$$
$$\le \frac{1}{\sqrt{2\pi } s} \int_\mathbb{R}[1 + |x|^q + s^{\frac{q}{2}}|u|^q](u^2+1)e^{\frac{-u^2}{2}}du. $$
Observe that both integrals 

$$C_1=\int_\mathbb{R}(u^2+1)e^{\frac{-u^2}{2}}du,\quad C_2 = \int_\mathbb{R}|u|^q (u^2+1)e^{\frac{-u^2}{2}}
du.$$
are finite for every $q>0$. Then, 
$$|\partial^2_x \big(P_sg\big)(t,x)|\lesssim \frac{1}{s}\Big[ (1+ |x|^q)C_1 + s^{\frac{q}{2}} C_2 \Big]\lesssim_T \frac{1}{s}(1 + |x|^q),$$
for every $t,s \in [0,T]$ and $x \in \mathbb{R}$. 
\end{proof}

Observe that for each $v < t$, $\varphi^{(1)}_Y(t,v)$ is an $\mathcal{F}_v
$-measurable square-integrable random variable, so that there exists $\varphi^{(2)}_Y$ realizing 
$$\varphi^{(1)}_Y(t,v) = \mathbb{E}[\varphi^{(1)}_Y(t,v)]+ \int_0^v \varphi^{(2)}_Y(t,v; r)dW_r.$$
Moreover, $x\mapsto \partial_x (P_sg)(t,x) \in C^1$ so that $\varphi^{(1)}_Y(t,v) \in \mathbb{D}^{1,2}$ for $v< t$, where 

$$D_r \varphi^{(1)}_Y(t,v) = \partial^2_x \big(P_{\theta(v,t)} g) (t,\mathbb{E}^v[B_t])K(t,r)K(t,v),$$
for $ r < v < t$. 
\begin{lemma}
If $g$ is a Borel function with polynomial growth, then 

\begin{eqnarray*}
\varphi^{(2)}_Y(t,v; r)&=& \mathbb{E}^r\big[ \partial^2_x \big(P_{\theta(v,t)} g) (t,\mathbb{E}^v[B_t])\big]K(t,r)K(t,v)\\
& = & \partial^2_x \big(P_{\theta(r,t)} g) (t,\mathbb{E}^r[B_t])K(t,r)K(t,v),
\end{eqnarray*}
for every $r < v< t\le T$.
\end{lemma}
\begin{proof}
Fix $r < v < t$. The assertion is clearly true if $g$ is $C^2$ in the spatial variable. Indeed, in this case we observe

\begin{eqnarray*}
\varphi^{(1)}_Y(t,v)= \big(P_{\theta(v,t)} \partial_x g\big) (t,\mathbb{E}^v[B_t]) K(t,v)= \partial_x \big(P_{\theta(v,t)} g\big) (t,\mathbb{E}^v[B_t]) K(t,v)
\end{eqnarray*}
and the Malliavin derivative on the Gaussian space of $W$ is 

\begin{eqnarray*}
D_r \varphi^{(1)}_Y(t,v) &=& \partial^2_x \big(P_{\theta(v,t)} g\big) (t,\mathbb{E}^v[B_t]) K(t,r)K(t,v)\\
&=& \big(P_{\theta(v,t)} \partial^2_x g\big) (t,\mathbb{E}^v[B_t]) K(t,r)K(t,v)\\
&=& \mathbb{E}^v \big[ \partial^2_x g (t,B_t) \big]K(t,r)K(t,v).
\end{eqnarray*}
As a result, 

\begin{eqnarray*}
\varphi^{(2)}_Y(t,v; r)&=&\mathbb{E}^r[D_r \varphi_Y(t,v)]= \mathbb{E}^r[ \partial^2_x \big(P_{\theta(v,t)} g\big) (t,\mathbb{E}^v[B_t])] K(t,r)K(t,v)\\
&=& \mathbb{E}^r \big[ \partial^2_x g(t,B_t)\big] K(t,r)K(t,v)\\
&=&  \big(P_{\theta(r,t)} \partial^2_x g\big) (t,\mathbb{E}^r[B_t]) K(t,r)K(t,v)\\
&=& \partial^2_x  \big(P_{\theta(r,t)} g\big) (t,\mathbb{E}^r[B_t]) K(t,r)K(t,v).
\end{eqnarray*}

Take a standard mollifier $\eta_{\frac{1}{n}}$ (unit mass, compact support) and set
$$g_n(t,y)=\int_\mathbb{R}g(t,y-z)\eta_{\frac{1}{n}}(z)dz.$$
We have  

\begin{equation}\label{exch1}
\mathbb{E}^r \big[ \partial^2_x \big(P_{\theta(v,t)} g_n\big)(t,\mathbb{E}^v[B_t])\big] = \partial^2_x \big( P_{\theta(r,t)} g_n \big)(t,\mathbb{E}^r[B_t]),
\end{equation}
for every $n\ge 1$. Then,
using $\int\eta_{n^{-1}}dy=1$ and the compact support of $\eta_{n^{-1}}$, we have 

$$|g_n(t,y)|\le \int C(1+|y-z|^m)\eta_{n^{-1}}(z)dz\le C'(1+|y|^m),$$
for some $C'$ independent of $n$. So $\{g_n(t,\cdot); n\ge 1\}$ is uniformly of polynomial growth in the spacial variable. Since $g(t,\cdot)$ is locally integrable, then standard approximation by mollifiers yields

$$g_n(t,y)\rightarrow g(t,y)\quad \text{for a.e}~y \in \mathbb{R},$$
as $n\rightarrow + \infty$. Now, fix $s>0, (t,x) \in [0,T]\times \mathbb{R}$. Write 

$$L_s(u)=\Big( \frac{u^2}{s^2} -\frac{1}{s}\Big)\frac{1}{\sqrt{2\pi s}}e^{\frac{-u^2}{2s}}.$$
Hence, $|L_s(u)|\le C_s(1+u^2)e^{\frac{-u^2}{(2s)}}$ and $L_s \in L^1(\mathbb{R})$ with all polynomial moments finite, where $C_s$ is a positive constant which depends on $s$. Hence, with the uniform bound from above, we have 

$$|g_n(t,y)L_s(y-x)|\le C'(1+|y|^m)|L_s(y-x)|=:J(y).$$
Letting $u=y-x$, 

$$\int_{\mathbb{R}}J(y)dy = \int_{\mathbb{R}}(1+|u+x|^m)|L_s(u)|du < \infty,$$
because $L_s(u)$ is Gaussian times a polynomial, so it has finite moments of all orders. Therefore, $J \in L^1(\mathbb{R})$ and the bound is independent of $n$. Bounded convergence theorem yields  

\begin{eqnarray*}
\lim_{n\rightarrow +\infty}\partial^2_x \big(P_sg_n\big)(t,x) &=& \lim_{n\rightarrow +\infty}\int_\mathbb{R}g_n(t,y)\Big[ \frac{(y-x)^2}{s^2} - \frac{1}{s} \Big]p_s(y-x)dy\\
&=& \int_\mathbb{R}g(t,y)\Big[ \frac{(y-x)^2}{s^2} - \frac{1}{s} \Big]p_s(y-x)dy.
\end{eqnarray*}
This shows that 

\begin{equation}\label{exch2}
\lim_{n\rightarrow +\infty}\partial^2_x \big( P_{\theta(r,t)} g_n \big)(t,\mathbb{E}^r[B_t]) = \partial^2_x \big( P_{\theta(r,t)} g \big)(t,\mathbb{E}^r[B_t])
\end{equation}
and 
\begin{equation}\label{exch3}
\lim_{n\rightarrow +\infty} \mathbb{E}^r \big[ \partial^2_x \big(P_{\theta(v,t)} g_n\big)(t,\mathbb{E}^v[B_t])\big] = \mathbb{E}^r \big[ \partial^2_x \big(P_{\theta(v,t)} g\big)(t,\mathbb{E}^v[B_t])\big]
\end{equation}
in $L^2(\mathbb{P})$. Identities (\ref{exch1}), (\ref{exch2}), (\ref{exch3}) and Lemma \ref{growthsec} allow us to conclude the proof. 
\end{proof}

We now check that $Y_t = g(t,B_t)$ satisfies Assumption \ref{I4}.

\begin{lemma}
If $g$ has polynomial growth, then $\varphi^{(2)}_Y$ satisfies Assumption \ref{I4}.   
\end{lemma}
\begin{proof}
By Lemma \ref{growthsec}, we know there exists $q>0$ such that 
$$|\partial^2_x (P_sg)(t,x)|\lesssim_q s^{-1} \big( 1 + |x|^q\big),$$
for every $s>0,t>0$ and $x \in \mathbb{R}$. Then, 

\begin{eqnarray*}
|\varphi^{(2)}_Y(t,v;r)|&=& |\partial^2_x \big(P_{\theta(r,t)} g) (t,\mathbb{E}^rB_t)K(t,r)K(t,v)|\\
&\lesssim& (t-r)^{-2H}(1+ \sup_{r < t\le T}|\mathbb{E}^r[B_t]|^q)(t-r)^{H-\frac{1}{2}}(t-v)^{H-\frac{1}{2}},
\end{eqnarray*}
so that 

$$|\varphi^{(2)}_Y(t,v;r)|\frac{\partial K}{\partial t}(t,v)\lesssim (t-r)^{-2H}(1+ \sup_{r < t\le T}|\mathbb{E}^r[B_t]|^q)(t-r)^{H-\frac{1}{2}}(t-v)^{2H-2}$$
for $r < v< t$. Then, 

$$\int_r^T\int_r^t  |\varphi^{(2)}_Y(t,v;r)|\frac{\partial K}{\partial t}(t,v) dv dt\lesssim (1+ \sup_{r < t\le T}|\mathbb{E}^r[B_t]|^q)(T-r)^{H-\frac{1}{2}} $$
a.s for every $r \in [0,T]$. This shows that Assumption \ref{I4} is fulfilled. 
\end{proof}

We now check Assumption \ref{I5}. Recall 

\begin{equation*}
 (P_{\theta(u,t)} g)(t,x)
 = \frac{1}{\sqrt{2\pi \theta(u,t)}}
    \int_{\mathbb{R}} g(t,z)
     \exp\!\left(-\frac{(z-x)^2}{2\theta(u,t)}\right)\,dz .
\end{equation*}
Observe $\mathbb E^u[B_t]=\int_0^u (t-v)^{H-\frac12}\,dW_v$ is centered Gaussian with variance 
\begin{equation}\label{eq:sigma}
 \sigma^2(u,t):=\int_0^u (t-v)^{2H-1}dv
 = \frac{t^{2H}-(t-u)^{2H}}{2H},
\end{equation}
for $u \le t$. Then, for any $x\in\mathbb R$,
\begin{equation}\label{eq:gauss-rep}
\partial_x(P_{\theta(u,t)} g)(t,x)
=\int_{\mathbb R} g(t,z)\,\frac{z-x}{\theta(u,t)}\,p_\theta(z-x)\,dz
=\frac{1}{\sqrt{\theta(u,t)}}\,
\mathbb E\!\left[g\!\left(t,x+\sqrt{\theta(u,t)}\,Z\right)Z\right],
\end{equation}
where $Z\stackrel{d}{=} N(0,1)$ and $p_\theta$ is the $N(0,\theta(u,t))$-density.

\begin{lemma}[Gaussian regression identity]\label{lem:reg}
Let $Y_{u,t}:=\mathbb{E}^u[B_t] +\sqrt{\theta(u,t)}\,Z$, where $Z\stackrel{d}{=} N(0,1)$ is independent of $\mathbb{E}^u[B_t]$ and $u < t$. Then
\begin{equation}\label{eq:keyid}
 \mathbb E\!\left[\partial_x(P_{\theta(u,t)} g)(t,\mathbb{E}^u[B_t])\right]
 \;=\; \frac{1}{\sigma^2(u,t)+\theta(u,t)}\;\mathbb E\!\left[Y_{u,t}\,g(t,Y_{u,t})\right],
\end{equation}
where $Y_{u,t}\stackrel{d}{=} N\!\big(0,\sigma^2(u,t)+\theta(u,t)\big)$, for $u < t$. 
\end{lemma}

\begin{proof}
For $u < t$, we set $Y_{u,t}=\mathbb{E}^u[B_t]+\sqrt{\theta(u,t)}Z$. The pair $(Z,Y_{u,t})$ is jointly Gaussian with
$\mathrm{Cov}(Z,Y_{u,t})=\sqrt{\theta(u,t)}$ and $\mathrm{Var}(Y_{u,t})=\sigma^2(u,t)+\theta(u,t)$. Hence, 

$$\mathbb E[Z\mid Y_{u,t}]=\frac{\sqrt{\theta(u,t)}}{\sigma^2(u,t)+\theta(u,t)}\,Y_{u,t}.$$ 
Taking expectations gives \eqref{eq:keyid}.
\end{proof}

\begin{lemma}\label{checkass5}
If $g$ has polynomial growth,
then writing $c_k:=\mathbb E|Z|^k$ for $Z\stackrel{d}{=} N(0,1)$, we have 
\begin{align*}
 &\int_0^T\!\!\int_0^t 
   \Big|\mathbb{E}\big[\partial_x(P_{\theta(u,t)} g)\big(t,\mathbb{E}^u[B_t]\big)\big]\Big|
   \,(t-u)^{2H-2}\,du\,dt \\[2mm]
 &\qquad\lesssim_{H,T,m}
 \Bigg[ c_1\sqrt{2H}\,\frac{T^H}{H} \;+\;
        c_{m+1}\,(2H)^{-\frac{m-1}{2}}\,
        \frac{T^{Hm+H}}{Hm+H}\Bigg] \;<\;\infty.
\end{align*}
\end{lemma}

\begin{proof}
  By assumption
\begin{equation} \label{eq:gPolGrowth}
  |g(t,x)|\lesssim_{m,T} (1 + |x|^m),
\end{equation}  
for every $(t,x) \in [0,T]\times \mathbb{R}$,
  
 Lemma~\ref{lem:reg} and \eqref{eq:gPolGrowth}
\begin{eqnarray*}
 \Big|\mathbb{E}\big[\partial_x(P_{\theta(u,t)} g)(t,\mathbb{E}^u[B_t])\big]\Big|
&\lesssim_{m,T}& \frac{1}{\sigma^2(u,t)+\theta(u,t)} \mathbb{E}\left[|Y_{u,t}|(1+|Y_{u,t}|^m)\right]\\
&=& \frac{1}{\sigma^2(u,t)+\theta(u,t)}\Big( c_1 (\sigma^2(u,t)+\theta(u,t))^{1/2}\\
&+&  c_{m+1} (\sigma^2(u,t)+\theta(u,t))^{\frac{m+1}{2}}\Big) \\
&=& \left( c_1 (\sigma^2(u,t)+\theta(u,t))^{-1/2} + c_{m+1}(\sigma^2(u,t)+\theta(u,t))^{\frac{m-1}{2}}\right),
\end{eqnarray*}
for every $u < t$. Using \eqref{eq:sigma} and $\theta(u,t)=\frac{1}{2H}(t-u)^{2H}$,
\(
 \sigma^2(u,t)+\theta(u,t) \ge \frac{t^{2H}}{2H},
\)
so
\begin{equation*}
 \Big|\mathbb{E}[\partial_x(P_{\theta(u,t)}) g)(t,\mathbb{E}^u[B_t])]\Big|
 \lesssim_{m,T,H} \left( c_1\sqrt{2H}\,t^{-H}
     + c_{m+1}\,(2H)^{-\frac{m-1}{2}}\,t^{H(m-1)}\right).
\end{equation*}
Therefore,
\begin{align*}
 &\int_0^T\!\!\int_0^t 
   \Big|\mathbb{E}[\partial_x(P_{\theta(u,t)} g)(t,\mathbb{E}^u[B_t])]\Big|
   (t-u)^{2H-2}\,du\,dt \\
 &\quad\lesssim_{m,T,H} \int_0^T
 \left( \left( c_1\sqrt{2H}\,t^{-H}
     + c_{m+1}\,(2H)^{-\frac{m-1}{2}}\,t^{H(m-1)}\right)\right)
 \left(\int_0^t (t-u)^{2H-2}\,du\right) dt \\
 &\quad= \int_0^T
 \left( c_1\sqrt{2H}\,t^{-H}
     + c_{m+1}\,(2H)^{-\frac{m-1}{2}}\,t^{H(m-1)}\right) t^{2H-1}\,dt \\
 &\quad= \left( c_1\sqrt{2H}\int_0^T t^{H-1}dt
     + c_{m+1}(2H)^{-\frac{m-1}{2}}\int_0^T t^{Hm+H-1}dt\right) \\
 &\quad= 
 \Bigg[ c_1\sqrt{2H}\,\frac{T^H}{H}
     + c_{m+1}(2H)^{-\frac{m-1}{2}}\frac{T^{Hm+H}}{Hm+H}\Bigg] <\infty,
\end{align*}
for $1 > H>\tfrac12$ and $m\ge 0$.
\end{proof}

Summing up the above result, we obtain the following corollary of Theorem \ref{Koperator}.  

\begin{proposition}
If $g:[0,T] \times \mathbb{R}\rightarrow \mathbb{R}$ is a Borel function with polynomial growth, then $Y_\cdot =g(\cdot,B_\cdot)$ belongs to the domain of the forward stochastic integral. The operator $\mathcal{K}Y$ is expressed as (\ref{calKop}), where 

$$\mathbb{E}^v[Y_t] = \big(P_{\theta(v,t)}f\big) \big(t, \mathbb{E}^v[B_t]\big),\quad \varphi^{(1)}_Y(t,v) = \partial_x \big(P_{\theta(v,t)} g) (t,\mathbb{E}^v[B_t])K(t,v),$$

$$
\varphi^{(2)}_Y(t,v; r)=\partial^2_x \big(P_{\theta(r,t)} g) (t,\mathbb{E}^r[B_t])K(t,r)K(t,v),
$$
for $r < v< t\le T$.
\end{proposition}

\begin{remark}
In general, if $g$ is a Borel function with polynomial growth, then $g(\cdot,B_\cdot)$ does not need to belong to the domain of the adjoint of the gradient operator associated with the Gaussian space of $B$.  
\end{remark}

\section{Path-dependent cases}\label{S4}
In this section, we illustrate Theorem \ref{Koperator} with some examples of adapted processes as integrands. The goal is \textit{not} the presentation of the most general conditions on each class of examples but rather to give a short discussion of stochastic integration beyond the class of the state-dependent cases. Then, for the sake of simplicity, in the sequel, we mostly assume strong integrability conditions. The reader interested in covering the examples of this section under weaker integrability conditions can easily follow the steps below with the obvious modification.   

\subsection{Fractional Martingales} For a given constant $\alpha> -\frac{1}{2}$, let 

$$Y_t = \int_0^t (t-s)^\alpha Z_sdW_s,$$ 
for an $\mathbb{F}$-adapted process $Z$. This type of processes was studied by many authors, see e.g. \cite{hu2009}. We assume that


\begin{equation}\label{cond6}
r\mapsto \| Z_r\|_{L^2(\Omega)}
\end{equation} 
is bounded over $[0,T]$. 

\begin{proposition}
If $\alpha > \frac{1}{2}-H$ and (\ref{cond6}) is fulfilled, then $Y$ satisfies the conditions of Corollary \ref{CorMTH}, where the operator $\mathcal{K}Y$ is expressed in terms of 

$$\mathbb{E}^v[Y_t] = \int_0^v (t-s)^\alpha Z_sdW_s,\quad \varphi^{(1)}_Y(t,v) = (t-v)^{\alpha} Z_v$$
and 

$$\varphi^{(2)}_Y(t,v;r) = (t-v)^{\alpha} \varphi^{(1)}_Z(v,r),$$
for $r < v< t\le T$.  
\end{proposition}
\begin{proof}

  Let $A= \{ (r,x_1,x_2); 0 < r < x_1\wedge x_2, x_1\neq x_2 \}.$ We start by checking
  Assumption \ref{I3}. For this purpose, we need to estimate
$$(x_1-r)^{\alpha}(x_2-r)^\alpha\| Z_r\|_{L^2(\Omega)}  \sup_{\epsilon \in (0,1)}\| \mathbf{D}^\epsilon_{x_1,x_2;r}B\|_{L^2(\Omega))}; \quad r < x_1\wedge x_2, x_1 \neq x_2.$$ 
From Remark \ref{supq2}, we know that 

\begin{eqnarray*}
\sup_{\epsilon \in (0,1)}\big\|\mathbf{D}^\epsilon_{x_1,x_2;r}B\big\|^2_{L^2(\Omega)}&\le& \big\|\mathbf{D}_{x_1,x_2;r}B\big\|^2_{L^2(\Omega)}\\
&=& \int_0^r (x_1-s)^{2(H-1)}(x_2-s)^{2(H-\frac{3}{2})}ds,
\end{eqnarray*}
for $r < x_1\wedge x_2 \le T$ with $x_1\neq x_2$. Then, 

$$\int_A (x_1-r)^{\alpha}(x_2-r)^\alpha \| Z_r\|_{L^2(\Omega)}  \sup_{\epsilon \in (0,1)}\| \mathbf{D}^\epsilon_{x_1,x_2;r}B\|_{L^2(\Omega)}drdx_1dx_2$$
$$\lesssim_{\|Z_\cdot\|_{L^2(\Omega)}} \int_A (x_1-r)^{\alpha}(x_2-r)^\alpha \Bigg(\int_0^r (x_1-s)^{2(H-1)}(x_2-s)^{2(H-\frac{3}{2})}ds\Bigg)^{\frac{1}{2}}drdx_1dx_2<\infty, $$
for $\alpha > -\frac{1}{2}$. See (\ref{prodhalf}) for details.  

Next, 
$$
\int_0^T\int_r^T (t-r)^{2\alpha}\mathbb{E}|Z_r|^2dtdr = \int_0^T\int_0^t (t-r)^{2\alpha}\mathbb{E}|Z_r|^2drdt   < \infty,
$$
for $2 > \frac{1}{H}$. Then Assumption \ref{I1} is fulfilled. 


Now, we observe (recall (\ref{mumeasure}) and (\ref{Jset})) that 

$$\int_J \big\| \mathbb{E}^{v_1\wedge v_2} \big[ \bar{\varphi}^{(1)}_Y(x_1,v_1) \big] \big\|_{L^2(\Omega)} \big\| \mathbb{E}^{v_1\wedge v_2} \big[ \bar{\varphi}^{(1)}_Y(x_2,v_2) \big] \big\|_{L^2(\Omega)}\mu(dv_1dv_2dx_1dx_2)$$
$$= \int_J (x_1-v_1)^{\alpha}  (x_2-v_2)^{\alpha}\big\| \mathbb{E}^{v_1\wedge v_2} \big[ Z_{v_1} \big] \big\|_{L^2(\Omega)} \big\| \mathbb{E}^{v_1\wedge v_2} \big[ Z_{v_2} \big] \big\|_{L^2(\Omega)}\mu(dv_1dv_2dx_1dx_2)$$
$$\lesssim_{\|Z_\cdot\|_{L^2(\Omega)}} \int_J (x_1-v_1)^{\alpha}  (x_2-v_2)^{\alpha}\mu(dv_1dv_2dx_1dx_2)< \infty,$$
if $\alpha > \frac{1}{2}-H$ (see (\ref{prodhalf})). Then, by Lemma \ref{bassa3}, Assumption
\ref{I4} is fulfilled. Next,

$$\int_0^T \int_0^t\Big|\mathbb{E}[\varphi^{(1)}_Y(t,u)]\Big| \frac{\partial K}{\partial t}(t,u)dudt$$
$$ = \int_0^T \int_0^t (t-u)^{\alpha}|\mathbb{E}[Z_u]| \frac{\partial K}{\partial t}(t,u)dudt\lesssim_{H,\| Z_\cdot\|_{L^1(\Omega)}}\int_0^T \int_0^t (t-u)^{\alpha+ H-\frac{3}{2}}dudt< \infty.$$
This concludes the proof.  
\end{proof}

\subsection{Integral process} Let $Y$ be a square-integrable process in the domain of the forward integral and we set 

$$\mathcal{I}Y_t:= \int_0^t Y_sd^-B_s.$$ 
We want to investigate the existence of 

$$\int_0^T \mathcal{I}Y_td^-B_t$$
so that

\begin{eqnarray*}
\varphi^{(1)}_{\mathcal{I}Y}(t,r)=\mathcal{K}Y(t,r)&=& \int_r^t \Bigg\{\mathbb{E}^r[Y_u] \frac{\partial K}{\partial u} (u,r) + \varphi_Y^{(1)}(u,r)\mathcal{D}_{r;u}B\\
\nonumber&+& \int_r^u \varphi^{(2)}_Y(u,v; r) \frac{\partial K}{\partial u}(u,v)dv\Bigg\}du,
\end{eqnarray*}
for $0\le r< t$. We need to check Assumptions \ref{I1}, \ref{I3}, \ref{I4} and \ref{I5}.

 We first treat Assumption \ref{I3} and we choose as conjugate exponents $p=q=2$. There are nine components in $\|\varphi^{(1)}_{\mathcal{I}Y}(x_1,r)\varphi^{(1)}_{\mathcal{I}Y}(x_2,r)\|_{L^p(\Omega)}$ which we split it into three blocks. 

\

\text{BL 1:}

\begin{small}
$$\int_r^{x_1} \mathbb{E}^r[Y_u] \frac{\partial K}{\partial u} (u,r)du \int_r^{x_2} \mathbb{E}^r[Y_u] \frac{\partial K}{\partial u} (u,r)du+ \int_r^{x_1} \mathbb{E}^r[Y_u] \frac{\partial K}{\partial u} (u,r)du \int_r^{x_2}\varphi_Y^{(1)}(u,r)\mathcal{D}_{r;u}Bdu$$
$$+ \int_r^{x_1} \mathbb{E}^r[Y_u] \frac{\partial K}{\partial u} (u,r)du \int_r^{x_2}\int_r^u \varphi^{(2)}_Y(u,v; r) \frac{\partial K}{\partial u}(u,v)dvdu=\sum_{n=1}^3 I_n(r,x_1,x_2)$$
\end{small}

\

\text{BL 2:}

\begin{small}
$$\int_r^{x_1}\varphi_Y^{(1)}(u,r)\mathcal{D}_{r;u}Bdu\int_r^{x_2} \mathbb{E}^r[Y_u] \frac{\partial K}{\partial u} (u,r)du 
+ \int_r^{x_1}\varphi_Y^{(1)}(u,r)\mathcal{D}_{r;u}Bdu \int_r^{x_2}\varphi_Y^{(1)}(u,r)\mathcal{D}_{r;u}Bdu$$
$$+ \int_r^{x_1}\varphi_Y^{(1)}(u,r)\mathcal{D}_{r;u}Bdu \int_r^{x_2}\int_r^u \varphi^{(2)}_Y(u,v; r) \frac{\partial K}{\partial u}(u,v)dvdu=\sum_{n=1}^3 J_n(r,x_1,x_2)$$
\end{small}

\text{BL 3:}

\begin{small}
$$\int_r^{x_1}\int_r^u \varphi^{(2)}_Y(u,v; r) \frac{\partial K}{\partial u}(u,v)dvdu\int_r^{x_2} \mathbb{E}^r[Y_u] \frac{\partial K}{\partial u} (u,r)du $$
$$+  \int_r^{x_1}\int_r^u \varphi^{(2)}_Y(u,v; r) \frac{\partial K}{\partial u}(u,v)dvdu\int_r^{x_2}\varphi_Y^{(1)}(u,r)\mathcal{D}_{r;u}Bdu$$
$$+ \int_r^{x_1}\int_r^u \varphi^{(2)}_Y(u,v; r) \frac{\partial K}{\partial u}(u,v)dvdu \int_r^{x_2}\int_r^u \varphi^{(2)}_Y(u,v; r) \frac{\partial K}{\partial u}(u,v)dvdu=\sum_{n=1}^3 L_n(r,x_1,x_2).$$
\end{small}

To simplify the presentation, let us assume the triple $(Y,\varphi^{(1)}_Y,\varphi^{(2)}_Y)$ is bounded. We claim that 
\begin{equation}\label{claimb} 
(r,x_1,x_2)\mapsto \|\varphi^{(1)}_{\mathcal{I}Y}(x_1,r)\varphi^{(1)}_{\mathcal{I}Y}(x_2,r)\|_{L^2(\Omega)} = \|\mathcal{K}Y(x_1,r)\mathcal{K}Y(x_2,r)\|_{L^2(\Omega)}
\end{equation}
is bounded over the set $\{(r,x_1,x_2); r < x_1\wedge x_2\le T\}$.

Let us analyze BL 1. We notice that   
$$\mathbb{E}|I_1(r,x_1,x_2)|^2\le \| Y\|^4_\infty \Bigg( \int_r^{x_1} \int_r^{x_2} (u_1-r)^{H-\frac{3}{2}} (u_2-r)^{H-\frac{3}{2}}du_1du_2\Bigg)^2\lesssim_{H,T}\| Y\|^4_\infty.$$
Jensen's inequality yields
\begin{eqnarray*}
\mathbb{E}|I_2(r,x_1,x_2)|^2&\lesssim_{H,T}&\|Y\|^2_\infty \mathbb{E}\Bigg|\int_r^{x_2}\varphi_Y^{(1)}(u,r)\mathcal{D}_{r;u}Bdu\Bigg|^2\\
&\le& \|Y\|^2_\infty \| \varphi^{(1)}_Y\|^2_\infty \mathbb{E}\int_{r}^{x_2}|\mathcal{D}_{r;u}B|^2du\\
&=& \|Y\|^2_\infty \| \varphi^{(1)}_Y\|^2_\infty \int_r^{x_2}(u-r)^{2H-2}du\\
&\lesssim_{H,T}& \|Y\|^2_\infty \| \varphi^{(1)}_Y\|^2_\infty. 
\end{eqnarray*}
Moreover, 
\begin{eqnarray*}
\mathbb{E}|I_3(r,x_1,x_2)|^2&\le&\|Y\|^2_\infty\| \varphi^{(2)}_Y\|^2_\infty \Bigg( \int_r^{x_1}(u-r)^{H-\frac{3}{2}}du \Bigg)^2\\
&\times & \Bigg( \int_r^{x_2} \int_r^u (u-v)^{H-\frac{3}{2}}dvdu \Bigg)^2  \\
&\lesssim_{H,T}& \|Y\|^2_\infty \| \varphi^{(2)}_Y\|^2_\infty.
\end{eqnarray*}
Let us analyze BL 2. The analysis of $J_1$ is equal to $I_2$. Recall that for any $\gamma>1$, we have 
$$\mathbb{E}|\mathcal{D}_{r;u}B|^{2\gamma}\lesssim_\gamma \Bigg(\int_0^r (u-s)^{2H-3}ds\Bigg)^{\frac{2\gamma}{2}}\lesssim_{H,T} (u-r)^{2\gamma(H-1)}.$$
Then, a simple application of Jensen and H\"older inequalities yield 
\begin{eqnarray*}
\mathbb{E}|J_2(r,x_1,x_2)|^2&\le& \| \varphi^{(1)}_Y\|^2_\infty \mathbb{E}\Bigg( \int_r^{x_1} |\mathcal{D}_{r;u}B|du \int_r^{x_2} |\mathcal{D}_{r;u}B|du\Bigg)^2\\
&=& \| \varphi^{(1)}_Y\|^2_\infty \mathbb{E}\Bigg( \int_r^{x_1}\int_r^{x_2} |\mathcal{D}_{r;u_1}B| |\mathcal{D}_{r;u_2}B|du_1du_2\Bigg)^2\\
&\le& \|\varphi^{(1)}_Y\|^2_\infty  \mathbb{E}\int_r^{x_1}\int_r^{x_2} |\mathcal{D}_{r;u_1}B|^2 |\mathcal{D}_{r;u_2}B|^2du_1du_2\\
&\le& \|\varphi^{(1)}_Y\|^2_\infty \int_r^{x_1}\int_r^{x_2} \|\mathcal{D}_{r;u_1}B\|^2_{L^{2\alpha}(\Omega)} \|\mathcal{D}_{r;u_2}B\|^2_{L^{2\beta}(\Omega)}du_1du_2\\
&\lesssim_{H,T,\alpha,\beta}& \|\varphi^{(1)}_Y\|^2_\infty \int_r^{x_1}\int_r^{x_2} (u_1-r)^{2H-2} (u_2-r)^{2H-2}  du_1du_2\\
&\lesssim_{H,T,\alpha,\beta}& \|\varphi^{(1)}_Y\|^2_\infty. 
\end{eqnarray*}
Moreover, 
\begin{eqnarray*}
\mathbb{E}|J_3(r,x_1,x_2)|^2&\lesssim_{H,T}& \| \varphi^{(1)}_Y\|^2_\infty \| \varphi^{(2)}_Y\|^2_\infty   \int_r^{x_1} \mathbb{E}|\mathcal{D}_{r;u}B|^2du\\
&\lesssim_{H,T}&\| \varphi^{(1)}_Y\|^2_\infty \| \varphi^{(2)}_Y\|^2_\infty. 
\end{eqnarray*}
Let us analyze BL 3. The analyses of $L_2$ and $L_1$ are equal to $J_3$ and $I_3$, respectively. It remains to analyze $L_3$ and this can be estimated as follows:
\begin{eqnarray*}
\mathbb{E}|L_3(r,x_1,x_2)|^2&\lesssim_{H,T}& \| \varphi^{(2)}_Y\|^4_\infty. 
\end{eqnarray*}
All the above estimates hold uniformly w.r.t. $r < x_1\wedge x_2 \le T$. Summing all the above steps of the analysis of BL 1, BL 2 and BL 3, we conclude (\ref{claimb}). By using the fact we can take $2 < \frac{2}{\frac{3}{2}-H}$ in Remark \ref{intDeps}, we then conclude
to the validity of Assumption \ref{I3}. 



In order to check Assumption
\ref{I4}, we make use of Lemma \ref{bassa3}. A totally similar analysis to the proof of the claim (\ref{claimb}) yields that 
$$(v_1,v_2,x_1,x_2)\mapsto \big\|\overline{\mathcal{K}Y} (x_1,v_1)\big\|_{L^2(\Omega)} \big\|\overline{\mathcal{K}Y} (x_2,v_2)\big\|_{L^2(\Omega)}$$
is bounded over $J$, where $\overline{\mathcal{K}Y} (x_1,v_1) = \mathcal{K}Y (x_1,v_1) - \mathbb{E}[\mathcal{K}Y (x_1,v_1)]$. Since $\mu$ is a finite positive measure, we then have  
$$\int_{J} \big\|\overline{\mathcal{K}Y} (x_1,v_1)\big\|_{L^2(\mathbb{P})} \big\|\overline{\mathcal{K}Y} (x_2,v_2)\big\|_{L^2(\mathbb{P})}\mu(dv_1dv_2dx_1dx_2) < \infty,$$
where $J$ and $\mu$ are given by (\ref{Jset}) and (\ref{mumeasure}), respectively. Again the boundedness assumption on $(Y, \varphi^{(1)}_Y, \varphi^{(2)}_Y)$ and Theorem \ref{Koperator} imply 
$$\int_0^T \text{Var}\Bigg(\int_0^tY_sd^-B_s\Bigg)dt=\int_0^T\mathbb{E}\int_r^t |\mathcal{K}Y(t,r)|^2drdt < \infty.
$$
This implies that Assumption \ref{I1} is fulfilled. Similarly, 
$$\int_0^T \int_0^t\Big|\mathbb{E}[\varphi^{(1)}_{\mathcal{I}Y}(t,u)]\Big| \frac{\partial K}{\partial t}(t,u)dudt=\int_0^T \int_0^t \Big|\mathbb{E}\big[\mathcal{K}Y(t,u)\big]\Big| \frac{\partial K}{\partial t}(t,u)dudt<\infty.$$
Then, we arrive at the following result. 

\begin{proposition}\label{lastprop}
Assume that $(Y,\varphi^{(1)}_Y,\varphi^{(2)}_Y)$ are bounded. Then, $\mathcal{I}Y_t= \int_0^t Y_sd^-B_s$ satisfies 
\begin{eqnarray*}
\int_0^T \mathcal{I}Y_t d^-B_t &=& \int_0^T \int_0^s \mathbb{E}[\mathcal{K}Y(s,u)]\frac{\partial K}{\partial s} (s,u)duds\\
\nonumber&+& \int_0^T \mathcal{K}\mathcal{I}Y(T,r)dW_r,
\end{eqnarray*}
where 
\begin{eqnarray*}
\mathcal{K}\mathcal{I}Y(T,r)&=& \int_r^T \Bigg\{\mathbb{E}^r[\mathcal{I}Y_t] \frac{\partial K}{\partial t} (t,r) + \mathcal{K}Y(t,r)\mathcal{D}_{r;t}B\\
\nonumber&+& \int_r^t \varphi^{(2)}_{\mathcal{I}Y}(t,v; r) \frac{\partial K}{\partial t}(t,v)dv\Bigg\}dt,
\end{eqnarray*}
for $0\le r< T$, where $\varphi^{(1)}_{\mathcal{I}Y} = \mathcal{K}Y$ and $\varphi^{(2)}_{\mathcal{I}Y}$ is the second martingale derivative of $\mathcal{I}Y$. 
\end{proposition}

\begin{remark}
One typical example of process $Y$ satisfying the assumption of Proposition \ref{lastprop} is given by $Y = f(B)$, where $f \in C^2_b$ and bounded. 
\end{remark}

\noindent \textbf{ACKWNOLEDGMENTS.}  The research of AO was partially supported by Projeto Universal CNPq 408884/2023-1.
The research of FR was also partially supported by the  ANR-22-CE40-0015-01 project (SDAIM).

\bibliographystyle{acm}
\begin{quote}
\bibliography{refORV.bib}
\end{quote}


\end{document}